\documentclass[a4paper]{amsart}
\usepackage{amssymb}
\usepackage{verbatim}
\usepackage{amscd}
\def\card{{{\operatorname{card}}}}

\numberwithin{equation}{section}
\theoremstyle{plain}
\newtheorem{theorem}[equation]{Theorem}
\newtheorem{corollary}[equation]{Corollary}
\newtheorem{proposition}[equation]{Proposition}
\newtheorem{lemma}[equation]{Lemma}
\newtheorem*{(DQ1)}{(DQ1)}
\theoremstyle{definition}

\theoremstyle{remark}

\begin{document}
\title{On Subshift Presentations}
\author{Wolfgang Krieger}
\begin{abstract} We consider partitioned graphs, by which we mean
finite  
directed graphs 
with a partitioned  edge set $
{\mathcal E} ={\mathcal E}^-  \cup{\mathcal E}^+$.
With additionally given a relation $\mathcal  R$ between the edges in ${\mathcal E}^-$ and the edges in  $\mathcal  E^+ $,
and under the appropriate assumptions on ${\mathcal E}^-,{\mathcal E}^+    $ and $\mathcal  R$, denoting the vertex set of the graph by ${\frak P}$, we speak of an ${\mathcal  R}$-graph 
${\mathcal  G}_{\mathcal  R}({\frak P},{\mathcal E}^-,{\mathcal E}^+)   $. From  ${\mathcal  R}$-graphs  
${\mathcal G}_{\mathcal R}({\frak P},{ \mathcal E}^-,{ \mathcal E}^+)   $  we construct  semigroups (with zero) ${\mathcal S}_{\mathcal R}({\frak P}, { \mathcal E}^-,{ \mathcal E}^+)   $ that we call  ${\mathcal R}$-graph semigroups. We write a list of conditions on a topologically transitive subshift with Property $(A)$ that together are sufficient for the subshift to have an ${\mathcal  R}$-graph semigroup as its associated semigroup.

Generalizing previous constructions, we describe a method of presenting subshifts 
by means of suitably structured finite labelled directed  graphs $( {\mathcal V}, \negthinspace \Sigma,\lambda \negthinspace)$ with vertex set 
${\mathcal V}$, edge set $\Sigma$, and a label map that asigns to the edges in $\Sigma$ labels in an ${\mathcal R}$-graph semigroup ${\mathcal  S}_{\mathcal  R}({\frak P} , 
{ \mathcal E}^-, { \mathcal E}^- )$. \negthinspace We denote the   presented subshift by $X{(\mathcal V}, \Sigma, \lambda)$ and  call $X({\mathcal V}, \Sigma, \lambda)$ an ${\mathcal  S}_{\mathcal  R}({\frak P}, { \mathcal E}^-, { \mathcal E}^- )$-presentation.

We introduce  a Property $(B)$ of subshifts that describes a relationship between contexts of admissible words of a subshift, and we introduce a Property $(c)$ of subshifts, that in addition describes a relationship between the past and future contexts and the context of  admissible words of a subshift. Property $(B)$ and 
the simultaneous occurrence of properties $(B)$ and $(c)$ are invariants of topological conjugacy,

We consider  subshifts, in
which every admissible word  has a future context, that is compatible with 
its entire  past context. Such subshifts  we call right instantaneous. 
We introduce a Property $R I$ of subshifts, and we prove that this property  is a necessary and sufficient condition 
for a subshift to have  a right instantaneous presentation.  
We consider also  subshifts, in which 
every admissible word  has a future context, that is compatible with 
its entire  past context, and also a past context that is compatible with its entire future context.
Such subshifts we call bi-instantaneous. 
We introduce a Property $BI$ of subshifts, and we prove that this property  is a necessary and sufficient condition 
for a subshift to have  a  bi-instantaneous presentation. 

We define a subshift as strongly bi-instantaneous if it has for every sufficiently long admissible word $a$ an admissible word $c$, that is contained in both the future context of $a$ and the past context of $a$, and that is such that the word $ca$ is a word in  the future context of $a$, that is compatible with the entire past context of $a$, and the word $ac$ is a word in  the past context of $a$, that is compatible with the entire future context of $a$.
We show 
that a topologically transitive subshift with Property $(A)$, and associated semigroup  
a graph inverse semigroup $\mathcal S$, has an  ${\mathcal S}$-presentation, if and only if it has properties  $(c)$ and  $BI$, and  a strongly bi-instantaneous presentation, if and only if it has properties  $(c)$ and  $BI$,  and all of its bi-instantaneous presentations are strongly bi-instantaneous.

We construct a class of subshifts  with Property $(A)$, to which certain graph inverse semigroups  ${\mathcal  S}({\frak P} , { \mathcal E}^-\negthinspace, { \mathcal E}^+ )$ are associated, that do not have  
${\mathcal  S}({\frak P} , { \mathcal E}^-\negthinspace, { \mathcal E}^+ )$-presentations.

We associate to the  labelled directed graphs $({\mathcal V}, \Sigma,\lambda)$ topological Markov chains and Markov codes, and we derive an expression for the zeta function of $X({\mathcal  V}, \Sigma, \lambda)$ in terms of the zeta functions of the topological Markov shifts and the generating functions of the Markov codes.
 \end{abstract}
 
 \maketitle

\section{Introduction}
We study symbolic dynamical systems, that we construct by means of finite  directed graphs. We will use the symbol $s$\ ($t$) to denote the source (target) vertex of an edge, or of a finite path, or of a right (left) infinite path in a directed graph.

Let  there be given a finite  directed graph with vertex set ${\frak P}$ and edge set ${\mathcal E}$, and also  a partition 
$$
{\mathcal E} = {\mathcal E}^-  \cup{\mathcal E}^+.
$$
We set
\begin{align*}
&{\mathcal E}^- ({\frak q},{\frak r} ) = \{ e^- \in { \mathcal E}^- : s(e^-) ={ \frak q},\  t( e^-) =  {\frak r} \},
\\
&{ \mathcal E}^+({\frak q},{\frak r}) = \{ e^+ \in{ \mathcal E}^+ : s(e^+) ={ \frak r},\  t( e^+) =  {\frak q} \}, \qquad { \frak q},{\frak r} \in{ \frak P}.
\end{align*}
We assume that ${ \mathcal E}^- ({\frak q},{\frak r})  \neq \emptyset$ if and only if $  { \mathcal E}^+({\frak q},{\frak r}) \neq \emptyset, { \frak q},{\frak r} \in {\frak P}.$ We also assume, that the directed graph $({\frak P},{ \mathcal E}^-)$ is strongly connected, or, equivalently, that the directed graph $({\frak P},{\mathcal E}^+)$ is strongly connected. We call the structure $({\frak P},{ \mathcal E}^-,{\mathcal E}^+)$ a partitioned graph.
Let there also be given relations 
$$
{ \mathcal R}    ({\frak q},{\frak r}) \subset  { \mathcal E}^- ({\frak q},{\frak r})   \times   { \mathcal E}^+({\frak q},{\frak r}) , \qquad{\frak q},{\frak r} \in{ \frak P},
$$
and set
$$
{ \mathcal  R} = \bigcup_{{\frak q},{\frak r} \in \frak P}{ \mathcal  R}    ({\frak q},{\frak r}) .
$$
The structure that is given in this way, for which we use the notation 
$ { \mathcal  G}_{ \mathcal R}(\frak P,  { \mathcal  E}^-,{ \mathcal  E}^+   )$, we call an ${ \mathcal  R}$-graph.
From an ${ \mathcal  R}$-graph ${ \mathcal  G}_{{ \mathcal  R}}({\frak P},  
{ \mathcal  E}^-, { \mathcal  E}^+   )$  we  construct  a semigroup (with zero) 
 ${ \mathcal  S}_{{ \mathcal  R}}({\frak P},  { \mathcal  E}^-, { \mathcal  E}^+   )$  that contains idempotents $\bold 1_{{\frak p}},{ \frak p} \in {\frak P},$ and that has ${ \mathcal E}$ as a generating set. Besides $\bold 1_{{\frak p}}^2 = \bold 1_{{\frak p}},\frak p \in{ \frak P}$, the defining relations of ${ \mathcal  S}_{{ \mathcal  R}}({\frak P},   { \mathcal  E}^-,{ \mathcal  E}^+   )$ are:
\begin{align*}
&{\bold 1}_{{\frak q}} e^- = e^-{ \bold 1}_{{\frak r}} = e^-, \quad e^- \in  {\mathcal E}^- ({\frak q},{\frak r}), \\
&{\bold 1}_{{\frak r}} e^+ = e^+{ \bold 1}_{{\frak q}} = e^+, \quad e^+ \in {\mathcal E}^+ ({\frak q},{\frak r}), \quad { \frak q},{\frak r} \in{\frak P},
\end{align*}
$$
f^- g^+  = \begin{cases}{ \bold 1}_{{\frak q}}, &\text {if $ (  f^- ,  g^+ ) \in{ \mathcal R}({\frak q},{\frak r}) $,}\\
0, & \text{if $ (  f^- ,  g^+ )\notin {\mathcal R}({\frak q},{\frak r}), \quad f^- \in  {\mathcal E}^- ({\frak q},{\frak r}), g^+ \in { \mathcal E}^+({\frak q},{\frak r)},        
\
{\frak q},  {\frak r} \in{\frak P},
 $}
\end{cases}
$$
and 
$$
{\bold 1}_{{\frak q}}{\bold 1}_{{\frak r}}= 0,    \quad { \frak q},{\frak r} \in {\frak P},{\frak q} \neq{\frak r} .
$$
We call ${\mathcal S}_{{\mathcal R}}({\frak P},  { \mathcal E}^-, {\mathcal E}^+ )  $ an ${\mathcal R}$-graph semigroup.
The graph inverse semigroups (the generalized polycyclic semigroups, see \cite {AH} and \cite [Section 10.7]{L} and compare \cite {CK}) are a special case of ${\mathcal R}$-graph semigroups: The graph inverse semigroup of a finite directed 
graph with vertex set ${\frak P}$ and edge set ${\mathcal E}$ is obtained by taking a copy of the graph $({\frak P}, {\mathcal E})$ with vertex set ${\frak P}$ 
and edge set ${\mathcal E}^- = \{ e^- : e \in  {\mathcal E}^{\circ }  \}$ and a copy $({\frak P},{\mathcal E}^+)$  of the reversed graph of $({\frak P},{\mathcal E})$  and by 
constructing the ${\mathcal R}$-graph semigroup of the partitioned graph $({\frak P}, {\mathcal E}^- , {\mathcal E}^+)  $ with the relations
$$
{\mathcal R}({\frak q},{ \frak r}) = \{(e^-, e^+) : e \in {\mathcal E}  , s(e) = {\frak q}, t(e) = {\frak r} \}, \quad  { \frak q},{\frak r } \in{ \frak P}.
$$
The relation being understood we suppress the $\mathcal R$ and denote the graph inverse semigroup of the finite directed graph by ${\mathcal S}({\frak P},  { \mathcal E}^-, {\mathcal E}^+ )  $.
The Dyck inverse monoids (the polycyclic monoids) ${\mathcal D} _N, N> 1,$ \cite {NP} are obtained in this way from the one-vertex graph with $N$ loops. 
For the   ${\mathcal R}$-graph semigroups that are obtained from  a one-vertex graph see also 
 \cite [Section 4]{HK1}.
 
 In Section 2 we show that the isomorphism of ${\mathcal R}$-graph semigroups 
 ${\mathcal S}_{{\mathcal R}}({\frak P},  { \mathcal E}^-, {\mathcal E}^+ )  $ implies the isomorphism of the 
 ${\mathcal R}$-graphs  ${\mathcal G}_{{\mathcal R}}({\frak P},  { \mathcal E}^-, {\mathcal E}^+ )  $.

For a semigroup (with zero) $\mathcal S$, and  for $F \in  \mathcal S $ set
$$
\Gamma(F) = \{ (G^-, G^+) \in{\mathcal S} \times {\mathcal S} : G^-FG^+  \neq 0 \},
$$
and
$$
[F] = \{F^{\prime}\in{\mathcal S}: \Gamma (F^{\prime}) =  \Gamma (F)\}.
$$
The  set $ [{\mathcal S}]=\{[F]: F \in{\mathcal S}\}$ with the product given by
$$
[G][H]=[GH],\quad G, H \in {\mathcal S},
$$
is a semigroup (e.g.  see \cite[Section 2.2]{P}). 
In section 2  we also write a list of  conditions on a semigroup (with zero)  $\mathcal S$ that together are necessary and sufficient for the semigroup  to be an ${\mathcal R}$-graph semigroup,
such that
the projection of ${\mathcal S}$ onto $[{\mathcal S}]$ is an isomorphism.

In symbolic dynamics one studies subshifts $(X,S_{X})$, where $X$ a 
shift invariant closed subset of the shift space $\Sigma^{{\Bbb Z}}$  and $S_{X}$ is
the restriction of the shift on  $\Sigma^{{\Bbb Z}}$    
to $X$. For an  introduction to 
the theory of subshifts see \cite {Ki} and \cite {LM}. The study of topological conjugacy and  of  its invariants occupies a central place in symbolic dynamics.  Any subshift $ \widetilde{X} $ that is topologically conjugate to a given subshift $X$ is called a presentation of $X$. 
A Property $(A)$ of  subshifts, an invariant of topological conjugacy, was introduced in \cite {Kr2}, and for  a subshift $X$ with property 
$(A)$  a
semigroup (with zero) ${\mathcal S}(X)$ was constructed that 
is also an invariant of topological conjugacy.
Next to topological conjugacy flow equivalence  and its invariants are studied in symbolic dynamics. It was shown by Costa and Steinberg \cite{CS} that Property $(A)$ and the associated semigroup are also  invariants of flow equivalence.

In section 3 we  translate all but the last of the conditions on the list of Section 2  into conditions on a subshift that are invariant under topological conjugacy. For the last condition we obtain a possibly stronger version for subshifts, that is also invariant under topological conjugacy. For a subshift $X$ with property $(A) $  these conditions together imply that ${\mathcal S}(X)$ is an ${\mathcal R}$-graph semigroup.

We describe now a way to present subshifts by means of ${\mathcal R}$-graph semigroups. We follow here closely \cite [Section 3]  {HIK} where this method of presenting subshifts was introduced for the case of graph inverse semigroups of directed graphs in which every vertex has at least two incoming edges.
These presentations were introduced 
for the purpose of extending the criterion for the existence of an enbedding of an irreducible subshift of finite type into a Dyck shift \cite {HI} to a wider class of target shifts.
Given an ${\mathcal R}$-graph semigroup ${\mathcal S}_{{\mathcal R}}({\frak  P}, 
 {\mathcal E}^-,{\mathcal E}^+)$,  denote by ${\mathcal S}^-_{{\mathcal R}}({\frak  P},    {\mathcal E}^- )$
  (${\mathcal S}^+_{{\mathcal R}}(\frak P,    {\mathcal E}^+ )$) the subset of 
  ${\mathcal S}_{{\mathcal R}}({\frak  P},  
   {\mathcal E}^-,  {\mathcal E}^+   )$ that contains the non-zero elements of the 
   subsemigroup of ${\mathcal S}_{{\mathcal R}}({\frak  P},   {\mathcal E}^-, 
    {\mathcal E}^+   )$ that is generated by 
  $ {\mathcal E}^-$(   $ {\mathcal E}^+$  ), and 
consider a finite strongly connected labelled directed graph with vertex set ${\mathcal V}$ and edge set $\Sigma$, and a labeling map $\lambda$ that assign to every edge $\sigma \in \Sigma$ a label 
\begin{align*}
 \lambda(\sigma) \in  {\mathcal S}^-_{{\mathcal R}}({\frak  P},   {\mathcal E}^- ) \cup
  \{{ \bold 1}_{\frak p}:{\frak p} \in {\frak  P}\} \cup {\mathcal S}^+_{{\mathcal R}}({\frak  P},   {\mathcal E}^+) . \tag{G1}
\end{align*}
The label map $\lambda$ extends to finite paths $(\sigma_i)_{1 \leq i \leq I}, I \in \Bbb N,$ in the graph $({\mathcal V}, \Sigma)$ by
$$
\lambda((\sigma_i)_{1 \leq i \leq I}) = \prod_{1 \leq i \leq I}\lambda (\sigma_i).
$$
We denote for ${\frak p} \in {\frak  P}$ by ${\mathcal V}({\frak p}) $ the set of $V\in{\mathcal V}$ such that there is a cycle $(\sigma_i)_{1 \leq i \leq I},I \in \Bbb N,$ in the graph $({\mathcal V} , \Sigma  )$ from $V$ to $V$ such that
$$ 
\lambda(( \sigma_i)_{1 \leq i \leq I}) = {\bold 1}_{{\frak p}}. 
$$
We impose  conditions (G 2 - 5):

\bigskip
\noindent(G2)\quad\quad \quad\quad\quad\quad\quad\quad\quad\quad\quad\quad
$
{\mathcal V}({\frak p}) \neq \emptyset, \quad {\frak p}\in {\frak P},
$

\bigskip
\noindent(G3)\quad\quad\quad\quad\quad\quad\quad\quad \ 
$
 \{{\mathcal V}({\frak p}) :
{\frak p} \in{\frak P}\}
$ 
 is a partition of  ${\mathcal V},
$
\bigskip

\noindent(G4)
 For $ V  \in {\mathcal V}({\frak p}), {\frak p} \in{ \frak P}$,
 and for all edges $e$ that leave $V,{ \bold 1}_{\frak p} \lambda(e) \neq 0$, 
and for all edges $e$ that enter $V, \lambda(e){\bold 1}_{{\frak p}}  \neq 0$.

\bigskip
\noindent(G5)
For $f\in {\mathcal S}_{{\mathcal R}}({\frak  P}, 
 {\mathcal E}^-,{\mathcal E}^+),  {\frak q},  {\frak r} \in 
{\frak P}$, such that ${\bold 1}_{{\frak q}}  f{\bold 1}_{{\frak r}}  \neq 0$,  and for $   U  \in{ \mathcal V} {(\frak q}), W  \in{ \mathcal V}({\frak r})  $, there exists
a path $b$ in the labeled directed graph  $({\mathcal V},\Sigma, \lambda)$ from $U$ to $W$ such that 
 $\lambda(b) = f$.

A finite  labelled directed graph $({ \mathcal V},\Sigma, \lambda)$, that satisfies conditions 
(G 1 - 5), gives rise to a subshift  $X({ \mathcal V},\Sigma, \lambda)  $ that has as its language of admissible words the set of finite non-empty paths $b$ in  the graph $({ \mathcal V}, \lambda)$   such that $\lambda (b) \neq 0$.
We  call the subshift $X({ \mathcal V},\Sigma, \lambda)  $ an 
$S_{\mathcal R}({\frak P}, {\mathcal E}^-,  {\mathcal E}^+ )$-presentation, and we say that a subshift $X$ has an $S_{\mathcal R}({\frak P}, {\mathcal E}^-,  {\mathcal E}^+ )$-presentation,   if it is topologically conjugate to an  $S_{\mathcal R}({\frak P}, {\mathcal E}^-,  {\mathcal E}^+ )$-presentation. The 
$S_{\mathcal R}({\frak P}, {\mathcal E}^-,  {\mathcal E}^+ )$-presentations with a labeling map $\lambda$ such that
$$
\lambda(\sigma) \in  {\mathcal E}^- \cup \{{ \bold 1}_{\frak p}:{\frak p} \in {\frak  P}\} \cup {\mathcal E}^+, \qquad \sigma \in \Sigma,
$$
belong to the class of sofic-Dyck shifts as described in \cite{BBD1},\cite { BBD2}.
In \cite [Section 6] {HIK} and \cite [Section 4] {IK2} $\mathcal D_2$-presentations are described, that arise as intersections of $\mathcal D_2$  with subshifts of finite type.

The special cases of $S_{\mathcal R}({\frak P}, {\mathcal E}^-,  {\mathcal E}^+ )$-presentations that are obtained for ${ \mathcal V} ={\frak P}, \Sigma= {\mathcal E}, $ and with the identity map as the labeling map, 
are the ${\mathcal R}$-graph shifts (see \cite [Section 4]{HK1})  and \cite {HK2}).
In the case of directed graphs the ${\mathcal R}$-graph shifts are the Markov-Dyck shifts  \cite {M3},
and  in the case of one-vertex directed graphs the ${\mathcal R}$-graph shifts are the Dyck shifts \cite {Kr1}.
The ${\mathcal R}$-graph shifts belong to the class of FT-Dyck shifts as described in \cite { BBD3}. For ${\mathcal R}$-graph shifts see also \cite {Kr6}.
Adding a loop at each vertex ${\frak p} \in {\frak P}$ with label ${\bold 1}_{\frak p}$, one obtains what can be called the  ${\mathcal R}$-graph Motzkin shifts. In the case of directed graphs the ${\mathcal R}$-graph Motzkin shifts are the Markov-Motzkin shifts \cite [Section 4.1]{KM3},
and  in the case of one-vertex directed graphs the ${\mathcal R}$-graph Motzkin shifts are the Motzkin shifts \cite {M2, I}.

In section 4 we introduce
a Property $(B)$ of subshifts  that describes a relationship between contexts of admissible words of a subshift. In section 5 we introduce a Property $(c)$ of subshifts, that  describes a relationship between the past and future contexts and the context of  an admissible word of a subshift. 
One is lead to the formulation of  properties $(B)$ and $(c)$ by observing the behavior of the Dyck shifts and by abstracting  essential dynamical properties.
We prove that  Property $(B)$ is an invariant of topological conjugacy, and that
the simultaneous occurrence of properties $(B)$ and $(c)$ is also an invariant of topological conjugacy.
An ${\mathcal S}_{{\mathcal R}}({\frak P},  { \mathcal E}^-, {\mathcal E}^+ )  $-presentation that has Property $(A)$, has Property $(B)$ precisely if ${\mathcal S}_{{\mathcal R}}({\frak P},  { \mathcal E}^-, {\mathcal E}^+ )  $ is a graph inverse semigroup.

In section 6 we consider  subshifts, in
which every admissible word  has a future context, that is compatible with 
its entire  past context. Such subshifts  we call right instantaneous. 
We introduce a Property $R\negthinspace I$ of subshifts, and we prove that this property  is a necessary and sufficient condition 
for the subshift to have  a right instantaneous presentation.  
We consider also  subshifts, in which 
every admissible word  has a future context, that is compatible with 
its entire  past context, and also a past context that is compatible with its entire future context.
Such subshifts we call bi-instantaneous. 
We introduce a Property $B\negthinspace I$ of subshifts, and we prove that this property  is a necessary and sufficient condition 
for the subshift to have  a  bi-instantaneous presentation.  Sofic systems have bi-instantaneous presentations.

In Section 7 we introduce a notion of strong bi-instantaneity. 
We show 
that a topologically transitive subshift with Property $(A)$, and associated semigroup  
a graph inverse semigroup $\mathcal S(\frak P, \mathcal E^-, \mathcal E^+)$, has an 
$\mathcal S(\frak P, \mathcal E^-, \mathcal E^+)$-presentation, if and only if it has properties 
$(c)$ and  $B I$ and  a strongly bi-instantaneous presentation, if and only if it has properties $(c)$ and  $B I$,  and  a strongly bi-instantaneous presentations, and  all of its bi-instantaneous presentations are strongly bi-instantaneous.

 In section 8 we construct   for a graph $\mathcal G(\frak P, \mathcal E)$, such that for
$\frak q,\frak r \in \frak P  $,
 $\card (\mathcal E(\frak q , \frak r) ) \neq 1$, a subshift with property (A) and with the  graph inverse semigroup of  $\mathcal S(\frak P, \mathcal E^-, \mathcal E^+)$ as associated semigroup, that does not have an   $\mathcal S(\frak P, \mathcal E^-, \mathcal E^+)$-presentation.

In section 9, applying methods of Keller \cite {Ke}, we associate to the  labelled directed graphs $({\mathcal V}, \Sigma,\lambda)$ topological Markov chains and circular Markov codes, and we derive an expression for the zeta function of $X({\mathcal  V}, \Sigma, \lambda)$ in terms of the zeta functions of the topological Markov shifts and the generating functions of the Markov codes.

\section{ ${\mathcal R}$-graph semigroups}
Let ${\mathcal S}$ be a semigroup (with zero). We set
\begin{align*}
&{\mathcal C}_-(F) = \{G\in {\mathcal S}: GF \neq 0 \},\\
&{\mathcal C}_+(F) = \{G\in{\mathcal S}: FG \neq 0 \},\quad F \in {\mathcal S}.
\end{align*}
We denote by  ${\mathcal U}_{\mathcal S}$ the set of idemptents $U$ of ${\mathcal S}$, such that one has for 
$F_-\in{\mathcal C}_-(U)$ that $F_-U = F_- $ and  for $F_+\in{\mathcal C}_+(U)$ that  $UF_+ = F_+$.
We set
\begin{align*}
&{\mathcal S}^-(U) =
 \bigcap_{F_- \in{\mathcal C}_-(U)}
{\mathcal C}_+(F_-U), 
\quad U \in {\mathcal U}_{\mathcal S},
\\
&{\mathcal S}^-(V,W) ={\mathcal S}^-(V)W \setminus \{0\}, \quad V ,W\in {\mathcal U}_{\mathcal S},
\end{align*}
and
\begin{align*}
&{\mathcal S}^+(U) = \bigcap_{F_+ \in{\mathcal C}_+(U)}{\mathcal C}_-(UF_+), \quad U \in {\mathcal U}_{\mathcal S},
\\
&{\mathcal S}^+(V,W) =  W{\mathcal S}^+(V) \setminus \{0\}, \quad W, V \in {\mathcal U}_{\mathcal S}.
\end{align*}
We  define sub-semigroups (with zero) ${\mathcal S} ^-$ and  ${\mathcal S} ^+$ of ${\mathcal S}$ by setting
$$
{\mathcal S}^- = \bigcup_ {U \in{\mathcal U}_{\mathcal S}}{\mathcal S} ^- (U), \quad
{\mathcal S}^+ =  \bigcup_ {U \in{\mathcal U}_{\mathcal S}}{\mathcal S} ^+ (U). 
$$
We say that an element $F\in{\mathcal S}  ^- \setminus  {\mathcal U}_{\mathcal S}  $ is indecomposable in ${\mathcal S}  ^-$ if $F =GH, G, H \in {\mathcal S} ^- ,$ implies that $G$ or $H$ is in ${\mathcal U}_{\mathcal S} $. The indecomposable elements in ${\mathcal S}  ^+ $ are defined symmetrically. 

\begin{theorem}
 The isomorphism of the  ${\mathcal R}$-graph semigroups 
 ${\mathcal S}_{{\mathcal R}}({\frak P},  { \mathcal E}^-, {\mathcal E}^+ )  $ implies the isomorphism of the 
 ${\mathcal R}$-graphs  ${\mathcal G}_{{\mathcal R}}({\frak P},  { \mathcal E}^-, {\mathcal E}^+ )  $.
\end{theorem}
\begin{proof}
Let
$$
\mathcal S ={\mathcal S}_{{\mathcal R}}({\frak P},  { \mathcal E}^-, {\mathcal E}^+ ) .
$$
With vertices
$$
\frak p_i \in \frak P, \quad 0 \leq i \leq I, \quad I \in \Bbb Z_+,
$$
and edges
$$
e^+_i \in \mathcal E^+ (\frak p_{i-1}  ,  \frak p_i ), \quad I \geq i \geq 1,
$$
$$
e^-_i\in \mathcal E^- (\frak p_{i-1} , \frak p_i ), \quad 1 \leq i \leq I,
$$
such that
$$
( e^-_i , e^+_i   ) \in \mathcal R(\frak p_{i-1}  ,  \frak p_i ), \quad  1 \leq i \leq I,
$$
the normal forms of the idempotents of $\mathcal S$ are given by
$$
(\prod_{I\geq i \geq1} e^+_{i}){\bold 1}_{\frak p_0}(\prod_{1\leq i \leq I} e^-_{i}). 
$$
 From this expression it can be seen that 
 $$
 {\mathcal U}_{\mathcal S} = \{{ \bold 1}_{\frak p}:{\frak p} \in {\frak  P}\},
 $$
which then means that one can identify $\frak P$ with ${\mathcal U}_{\mathcal S}$, and then also 
$\mathcal E^-(\frak q, \frak r)$ with the set of indecomposable elements $f$ of $\mathcal S^-$, such that 
$
\bold 1_{\frak q}f \neq 0,f  \bold 1_{\frak r}\neq 0,
$
and 
$\mathcal E^+(\frak q, \frak r)$ with the set of indecomposable elements of $\mathcal S^+$, such that 
$
\bold 1_{\frak r}f \neq 0,f  \bold 1_{\frak q}\neq 0.
$
Also, for  $e^- \in\mathcal E^-(\frak q, \frak r), e^+ \in\mathcal E^+(\frak q, \frak r),\frak q, \frak r\in \frak P,  $ one has $ (  e^- ,   e^+ ) \in 
\mathcal R (\frak q, \frak r), $ if and only if $e^-    e^+ \neq 0$.
\end{proof}

We have a condition  (AP1) on a semigroup $\mathcal S$:

\bigskip

\noindent
$\bold{(AP1)}$
 ${\mathcal U}_{\mathcal S}$ is a finite set. 
 
 \medskip
 
We also have a condition (AP2) on a semigroup $\mathcal S$, that comes in two parts that are symmetric to one another, of which we write only one part:

\bigskip

\noindent
$\bold{(AP2-)}$ 
${\mathcal S}^-(V,W)
 \neq \emptyset, \quad V, W \in  {\mathcal U}_{\mathcal S}.$

\bigskip
We  also have conditions (AQ1$-$2) on a semigroup $\mathcal S$:  

\bigskip
\noindent
$\bold{(AQ1)}$
$
{\mathcal S}^-(V,U)  {\mathcal S}^+(W,U)  \subset {\mathcal S}^-(V)\cup {\mathcal S}^+(W), \quad U, V, W \in  {\mathcal U}_{\mathcal S}. 
$

\bigskip
\noindent
$\bold{(AQ2)}$
$
{\mathcal S} = \bigcup_{U \in {\mathcal U}_{\mathcal S}} {\mathcal S}  ^+ (U) {\mathcal S} ^- (U). 
$

\bigskip

We also have conditions (AQ3$-$6) on a semigroup $\mathcal S$, that similarly come in two parts, that  are symmetric to one another. We only write one part of these conditions:

\bigskip
 \noindent
 $\bold{(AQ3-)}$
For $U, V \in {\mathcal U}_{\mathcal S}$ one has that for $F^-\in {\mathcal S}^-(U,V) $ there exists an 
 $F^+\in{\mathcal S} ^+(U,V)$
  such that 
$$
F^-F^+= U.
$$

\bigskip
\noindent
$\bold{(AQ4-)}$
For $U,V,W \in {\mathcal U}_{\mathcal S}$ one has that for 
$$
F^- \in {\mathcal S}^- (U, V) \setminus \{ U\}, \qquad G^+\in 
{\mathcal S}^+ ( W,V) \setminus \{ V\},
 $$ 
 such that
$$
F^-G^+ \in{\mathcal S}^-(U, W),
$$
there exists an $H^- \in {\mathcal S}^-(W, V) \setminus \{ W\}$ such that $$F^-G^+H^-= F^-.$$

\bigskip
\noindent
$\bold{(AQ5-)}$
${\mathcal S}^-$ has finitely many indecomposable elements.

\bigskip
\noindent
$\bold{(AQ6-)}$ ${\mathcal S}^-$ is generated by its indecomposable elements.

\bigskip

Note that conditions  (AP2$-$)  and  (AQ3$-$) together imply Condition (AP2$+$),  and, symmetrically, conditions (AP2$+$)  and  
(AQ3$+$) together imply Condition (AP2$-$).

\bigskip
For an ${\mathcal R}$-graph  ${\mathcal G}_{\mathcal R}({\frak P},
{\mathcal E}^-, {\mathcal E}^- )$ we set
\begin{align*}
&\Omega^+(e^-) = \{e^+ \in  \mathcal E^+(\frak q,\frak r) : ( e^-,e^+ ) \in \mathcal R(\frak q  ,  \frak r ) \}, \quad  e^- \in \mathcal E^- (\frak q,\frak r), \\
&\Omega^-(e^+) = \{e^- \in   \mathcal E^-(\frak q,\frak r) : ( e^-,e^+ ) \in \mathcal R(\frak q  ,  \frak r ) \}, \quad  e^+ \in \mathcal E^+ (\frak q,\frak r) , \quad \frak q,\frak r \in \frak P.
\end{align*}

We introduce  condition (a)  on an ${\mathcal R}$-graph  ${\mathcal G}_{\mathcal R}({\frak P},
{\mathcal E}^-, {\mathcal E}^- )$, that  consists of two parts  that are symmetric to one another:

\begin{align*}
 \Omega^+(e^-) \neq \Omega^+(\widetilde e^-),\qquad \quad e^-,\widetilde e^-\in   \mathcal E^- (\frak q,\frak r),e^-\neq\widetilde e^-, \qquad \frak q,\frak r \in \frak P, 
 \tag {a $-$} 
\end{align*}
\begin{align*}
\Omega^-(e^+) \neq \Omega^-(\widetilde e^+), \qquad \quad e^+,\widetilde e^+\in   \mathcal E^+ (\frak q, \frak r),e^+\neq\widetilde e^+,\qquad
\frak q,\frak r \in \frak P.  
 \tag {a +} 
\end{align*}

The following theorem characterizes the $\mathcal R$-graph semigroups among the semigroups 
${\mathcal S}$ such that the projection of ${\mathcal S}$ onto $[{\mathcal S}]$ is an isomorphism. 
For the proof of the theorem  compare \cite [Section 3]{Kr4}.  

\begin{theorem} Let the  semigroup ${\mathcal S}$ be such that
the projection of ${\mathcal S}$ onto $[{\mathcal S}]$ is an isomorphism. Then ${\mathcal S}$ is an $\mathcal R$-graph semigroup 
if and only if ${\mathcal S}$ satisfies conditions (AP 1 $ -$ 2) and 
conditions (AQ 1 $-$ 6).
\end{theorem}

\begin{proof}  We prove sufficiency.
By (AP1) and (AQ1), (AQ5) and (AQ6) we can use a partitioned graph with vertex set ${\mathcal U}_{\mathcal S}$, and edge sets
$$
{\mathcal E}^- = \bigcup _{V, W \in {\mathcal U}_{\mathcal S} } {\mathcal E}^-(V, W), \quad {\mathcal E}^+ = \bigcup _{V, W \in {\mathcal U}_{\mathcal S} } {\mathcal E}^+(V, W),
$$
with the set of indecomposable elements in ${\mathcal S}^-(V, W)$ as the set 
 ${\mathcal E}^-(V, W)$, and the set of indecomposable elements in ${\mathcal S}^+(V, W)$ as the set ${\mathcal E}^+(V, W),V, W \in 
{\mathcal U}_{\mathcal S}.$ 

Let $V, W \in{\mathcal U}_{\mathcal S}$, and let $E^-\in {\mathcal S}^-(V, W)$ be indecomposable in  $ {\mathcal S}^-$. By $(AQ3-)$
there exists  an $E^+\in {\mathcal S}^+ (V, W)$ such that 
\begin{align*}
E^-E^+   = V. \tag {2.1}
\end{align*}
We prove that $E^+ $ is indecomposable.  Assume the contrary, and let $U \in{\mathcal U}_{\mathcal S} $ and
$$
F^+ \in {\mathcal S}^+(U, W),\quad  G^+ \in {\mathcal S}^+(V, U), 
$$
be such that
\begin{align*}
E^+ =F^+G^+, \quad G^+ \neq V. \tag {2.2}
\end{align*}
By $(AQ1)$ then either
$$
 E^-F^+ \in {\mathcal S} ^-(V  ) \setminus\{  V\},
$$
in which case there would exist by $(AQ4-)$ an $F^- \in{\mathcal S} ^-(U, W) \setminus \{ U\},$ such that 
$$
E^- =  E^-   F^+F^-,
$$
contradicting the indecomposability of $E^-$, or
$$
E^-F^+ \in  {\mathcal S} ^+(W),
$$
in which case
$$
E^-F^+ G^+  \in {\mathcal S} ^+(V)\setminus \{  V\},
$$
contradicting (2.1) and (2.2). The symmetric argument shows also
for $U, W \in{\mathcal U}_{\mathcal S} $, and for an $E^+\in {\mathcal S} ^+(U, W)$, that is  indecomposable in  $ {\mathcal S}^+$, 
that there exists  an $E^-\in {\mathcal S} ^+ (U, W)$, that is is  indecomposable in  $ {\mathcal S} ^-$, such that 
$$
E^-E^+   = U. 
$$
It follows from (AP2) that the directed graph  
$(  {\mathcal U}_ {\mathcal S}   , {\mathcal E} ^- )$ is strongly connected.
We define relations $ {\mathcal R}(V, W)\subset {\mathcal E}^-(V, W)
 \times {\mathcal E}  ^+(V, W) $ by 
$$
  {\mathcal R}(V, W)=
\{  (E^-, E^+)\in {\mathcal E} ^-(V, W)
 \times {\mathcal E}  ^+(V, W) : E^- E^+ \neq 0 \},  \quad V, W \in  
  {\mathcal U}_{\mathcal S} ,
$$
and set
$$
{\mathcal R} = \bigcup _{V, W \in  {\mathcal U}_{\mathcal S}}{\mathcal R}(V, W).
$$
By $(AQ4)$ one has 
\begin{multline*}
E^-E^+ \ = 
\ \begin{cases} U, &\text {if $U  =  W, \  (E^-, E^+) \in {\mathcal R}(U, V),$} \\
0,&\text {if $U  =  W, \ (E^-, E^+) \not  \in {\mathcal R}(U, V), $}\\
0,&\text {if $U  \neq  W,$}
\end{cases} 
\\
E^-
\in  {\mathcal E}^-(U, V), E^+\in{\mathcal E} ^+(W, V),\quad U, V, W \in  {\mathcal U}_
{\mathcal S}.
\end{multline*}

It follows from $(AQ2)$, $(AQ5)$,  and $(AQ6)$ that for $F\in {\mathcal S} $ there exist $I(-), I(+) \in \Bbb Z_+$ and
$$
U_{i_+}(+)\in{\mathcal U}_{\mathcal S} , \ \  I_+\geq i_+\geq 1,\quad U\in {\mathcal U}_
,\quad U_{i_-}(-)\in{\mathcal U}_{\mathcal S} ,  \ \ 
1\leq i_-\leq I_-,  
$$
and, setting
$$
U_0(+) = U_0(-)= U,
$$
also
$$
E^+_{i_+}\in   {\mathcal E} ^+(U_{i_+}  , U _{i_+-1} ),  \ \  I_+\geq i_+\geq 1, \quad
E^-_{i_-}\in  {\mathcal E} ^+(U_{i_--1}  , U _{i_-} ), \ \  
1\leq i_-\leq I_-,  
 $$
such that
\begin{align*}
F=
(\prod_{I_+\geq i_+ \geq1} E^+_{i_+})U(\prod_{1\leq i_- \leq I_-} E^-_{i_-}). \tag {2.3}
\end{align*}
It follows from the assumption that  the projection of ${\mathcal S} $ onto $[{\mathcal S} ]$ is an isomorphism, that the ${\mathcal R}$-graph ${\mathcal G} _{\mathcal R}({\mathcal U}_ {\mathcal S}   , {\mathcal E} ^-,  {\mathcal E} ^+   )$ satisfies  $(a)$, since for $E, \widetilde {E} \in{\mathcal E} ^-(U, V)$ one has that $\Omega^+(E ) =   \Omega^+( \widetilde {E})    $ would imply that $[E] =[\widetilde {E}] , U, V \in {\mathcal U} _{\mathcal S} $. 
Applying  $(a-)$ and $(a+)$ repeatedly, and keeping in mind that one has for $H \in {\mathcal S}, U \in  {\mathcal U}_{\mathcal S} $ that $[H] =[U]$ implies that $H = U$, one finds that the presentation (2.3) of $F$ is in fact unique.
\end{proof}

\section{Subshifts with property (A) to which ${\mathcal R}$-graph semigroups \\ are associated}

We introduce terminology and notation for subshifts. Given a subshift $X \subset \Sigma^
{{\Bbb Z}}$ we set
$$
x_{[i,k]}Ê=Ê(x_{j})_{i \leq j \leq k}, \quad  x \in X,  i,k\in {\Bbb Z} ,  i \leq k,
$$
and
$$
X_{[i,k]}Ê=Ê\{ x_{[i,k]} : x\in X \},\quad i,k\in {\Bbb Z} ,  i \leq k . 
$$
We use similar notation also for blocks,
$$
b_{[i',k']}Ê=Ê(b_{j})_{i' \leq j \leq k'}, \quad 	b\in X_{[i,k]}, \; i \leq i' \leq k' \leq k,
$$
and also if indices range in semi-infinite intervals. 
The symbol that denotes a block is also used to denote the word that is carried by the block. We identify the elements of $X_{(-\infty, 0)}$ with the left-infinite words that they carry, and we identify the elements of $X_{(0, \infty)}$ with the right-infinite words that they carry.
The set of admissible words of $X$ is denoted by $\mathcal L(X)$.
For the N- 
block 
system of a subshift $X \subset \Sigma^{{\Bbb Z}}$ we use the notations 
$$
x^{\langle(m,m+N]\rangle}Ê=Ê(x_{(i+m,i+m+N]})_{i\in{\Bbb Z}} , \qquad x\in X,
$$
$$
X^{\langle(m,m+N]\rangle}Ê=Ê\{ x^{\langle(m,m+N]\rangle} :  x\in X \}, 
\qquad
m \in {\Bbb Z} , N \in \Bbb N.
$$
We set
\begin{align*}
&\Gamma(a)Ê= \{(x^{-},x^{+}) \in  X_{(-\infty, 0)} \times X_{(0, \infty)}: (  x^{-} , a , x^{+} )  \in X\},
\\
&\Gamma_{n}^{+}(a)Ê= \{ b\in X_{(k,k+n]}: (a,b)\in X_{[i,k+n]}\} 
,\qquad n 
\in 
{\Bbb N} ,\\
&\Gamma_{\infty}^{+}(a)Ê= \{  y^{+}\in X_{(k,\infty )}:  (a, y^{+})\in 
X_{[i,\infty)}\} ,
\\
& \Gamma^{+}(a)Ê= 
 \bigcup_{n\in{\Bbb N}}\Gamma^{+}_{n}(a)Ê, \qquad a\in X_{[i,k]},	 i,k\in{\Bbb Z} ,  	i \leq k.
 \end{align*}
$ \Gamma^{-}$ has the time symmetric meaning. We set
\begin{align*}
&
\omega^{+}_{n}(a)Ê=
\bigcap_{x^{-} \in \Gamma_{\infty}^{-}(a)}
\{  b  \in X_{(k,k+n]}: (x^{-}, a,  b ) \in X_{( - \infty , k+n]}\},
\\
&\omega^{+}_{\infty}(a)Ê=Ê
\bigcap_{x^{-} \in \Gamma_{\infty}^{-}(a)}
\{ y^{+} \in X_{(k,\infty)}: (x^{-}, a, y^{+}) \in X\},
\\
& \omega^{+}(a)Ê= 
 \bigcup_{n\in{ \Bbb N}}\omega^{+}_{n}(a)Ê, \qquad a\in X_{[i,k]},	 i,k\in{ \Bbb Z} ,  	i \leq k.
\end{align*}		
$\omega^{-}$ has the time symmetric meaning. 

We recall that, given subshifts $X \subset \Sigma^{{\Bbb Z}},  \widetilde{X}\subset \widetilde{\Sigma}^{{\Bbb Z}}$,
and a  topological conjugacy $\varphi: X \rightarrow \widetilde{X} 
$, 
there is 
for some 
$L\in {\Bbb Z}_{+} $ a block mapping
$$
\Phi:  X_{[-L,L]} \rightarrow \widetilde{\Sigma } 
$$
such that
$$
\varphi (x)Ê=Ê(\Phi (x_{[i-L,i+L]}))_{i\in{\Bbb Z}} .
$$
We say then that $\varphi$ is given by $\Phi$, and we write
$$
\Phi( a)Ê=Ê(\Phi (a_{[j-L,j+L]}))_{i+L\leq j\leq k-L},	\quad a\in 
X_{[i,k]},  \quad i,k\in {\Bbb Z} ,	k - i \geq 2L,
$$
and use similar notation if indices range in semi-infinite intervals.

For a subshift $X\subset \Sigma ^{\Bbb Z}$ set
$$
 A_{n}^{-}(X) = \bigcap _{i\in {\Bbb Z}}
 \{x  \in X: 
x_{i} \in \omega^{+} (x_{[i - n, i)})\},\quad n \in {\Bbb N},
$$
and
$$
A^{-}(X) = \bigcup_{n \in{\Bbb N}} A_{n}^{-}(X) .
$$
Define $  A_{n}^{+}(X), n \in{ \Bbb N},$ and $ A^{+}(X)  $  symmetrically and set
$$
 A_{n}(X) = A_{n}^{-}(X)  \cap A_{ n}^{+}(X) ,\quad n \in {\Bbb N},
$$
and 
$$
A(X) = \bigcup_{n \in {\Bbb N}} A_{n}(X).
$$

Denote the set of periodic points in $A(X)$ by $P(A(X))$. 
We write for 
$q ,r  \in P(A(X)), q \succeq  r$, 
if there exists a point in $A(X) $ that is left asymptotic to  
$q$ and
right asymptotic to a point in the orbit of $r$. The equivalence relation that results from the preorder relation
$\succeq $ we write as
 $\approx$ and we denote the order relation that results from $\approx$ also by 
 $\succeq $.
We denote the set of $\approx$-equivalence classes by ${\frak P} (X)$. The order structure $({{\frak P}} (X)  ,\succeq )$ is invariantly associated to $X$ \cite {Kr2}. 
We denote by $Y(X)$ the set of points $x\in X$, such that there are $q, r \in P(A(X))$, such that  $x$ is left asymptotic to $q$ and right asymptotic to $r$.

We write a condition (DP0) on a subshift $X\subset \Sigma ^{\Bbb Z}$, that is invariant under topological conjugacy \cite {Kr2}:

\bigskip
\noindent
$\bold{(DP0)}$ $P(A(X))$ is dense in X. 

\bigskip
The following conditions (DP1) and (DP2) on a subshift $X\subset \Sigma ^{\Bbb Z}$ are invariant under topological conjugacy. Condition (DP1) is the  translation of Conditon (AP1) and Condition (DP2-) is the translation of Conditon (AP2-):

\bigskip
\noindent
$\bold{(DP1)}$ ${\frak P} (X)$ is a finite set.

\bigskip
\noindent
$\bold{(DP2-)}$ For $q, r\in P(A(X))$ there exists an $x \in A^-(X)$ that is left asymptotic to 
$q$ and right asymptotic to $r$.

\noindent

\bigskip

We introduce a Condition (DQ1) and a Condition (DQ2) on a subshift $X\subset \Sigma ^{\Bbb Z}$. Condition (DQ1)  is the translation of 
Condition (AQ1) and Condition (DQ2)  is the translation of 
Condition (AQ2):
 
\bigskip
\noindent
$\bold{(DQ1)}$ For $p \in P(A(X))$, there exists an $H \in {\Bbb N}$, such that, if
$$
y \in A^-(X) \cap   Y (X), \quad y_{(-H, \infty)} = p_{(-H, \infty)} , 
$$
$$
z \in A^+(X) \cap  Y (X),\quad z_{(- \infty,H]} = 
p_{(- \infty,H]}  , 
$$
 and
$$
(y_{(-\infty, 0]} , z_{(0, \infty)}) \in X,
$$
 then 
$$
(y_{(-\infty, 0]} , z_{(0, \infty)}) \in A^-(X) \cup  A^+(X). 
$$

\bigskip
\noindent
$\bold{(DQ2)}$ Given $q, r \in P(A(X))$, there exists an $H \in \Bbb N$, such that the following holds:
For $x \in   X$ and $K >H$, such that
$$
x_{(-\infty, -K]} = q_{(-\infty, 0]}, \quad x_{(K, \infty)} = r_{(0, \infty)},
$$
 and for $M \in \Bbb N$, there exist  $p \in P(A(X)), I, J > M,$ and
$$
y \in  A^+(X) \cap   Y (X) ,
 \quad z  \in A^-(X) \cap   Y (X) ,
$$
such that
$$
y_{(-\infty, -I]}   = q_{(-\infty, 0]}  , \quad y_{(M, \infty)}   = p_{(M, \infty)} ,
$$
$$
z_{(-\infty, M]} = p_{(-\infty, M]} ,    \quad z_{(J, \infty)}   = r_{( 0, \infty)} ,
$$
and such that
$$
\Gamma (x_{(-K -M, K+ M]}  ) =
 \Gamma((y_{(-I-H, M]}, z_{(M, J+ H]})  ) .
$$

\bigskip
Next we introduce a Condition (DQ3) and a Condition (DQ4) on a subshift $X\subset \Sigma ^{\Bbb Z}$. Again these  conditions come in two symmetric parts, of which we write only one.
Condition (DQ3) is the translation of Condition (AQ3) and Condition (DQ4) is the translation of Condition (AQ4):

\bigskip
\noindent
$\bold{(DQ3 - )}$ For $p, q \in P(A(X))$ and
$
x \in A^-(X) \cap  Y (X),$ that is left asymptotic to $p$, and such that
 $ x_{ (0, \infty)} =
p_{ (0, \infty)} ,
$
and for $M \in {\Bbb N}$ there exists a $y \in A^+(X) \cap  Y (X),$ that is right asymptotic to  a point in the orbit of $q$, such that 
$$
y_{ (- \infty, M]} =
p_{ (- \infty, M]} ,
$$
and such that
$$
(x_{ (- \infty, 0]} , y_{ (0.\infty)} ) \in A(X).
$$

\bigskip
\noindent
$\bold{(DQ4-)}$
For $q, r \in P(A(X))$ there exist an $H\in {\Bbb N}$ such that for $K_+, M> H$ and for 
$p \in P(A(X)), x \in A^-(X),y \in A^+(X)$ and $I, K_- \in \Bbb N,$ such that
$$
x_{(-\infty, - K_-]}  = p_{(-\infty, - K_-]} , \qquad  x_{(0, \infty)}  = q_{(0, \infty)} ,
$$
$$
y_{(-\infty, K_+]}= q_{(-\infty, K_+]} , \quad y_{(K_+ + I, \infty)}= r_{(K_+ + I,\infty)} ,
$$
and
$$
( x_{(-\infty, K_+]},  y_{(K_+, \infty)}) \in  A^-(X)  ,
$$
there exist a $z \in  A^-(X), $ and  $J, D\in{ \Bbb N}$, such that
$$
z_{(-\infty, K_++I + M]}= r_{(-\infty, 0]}, \qquad z_{(K_++I + M+J, \infty)}= q_{(0, \infty)} ,
$$
and
$$
 \Gamma(  x_{(- K_--D, K_+]}  , y_{(K_+, K_+ + I]}   ,  z_{(K_+ + I,K_+ + I+M+ J+D]}    )
=\Gamma(  x_{(- K_--D,K_+ + D ]}   ) .
$$

\bigskip

We call a point $x \in A^-(X)\cap Y(X)$ indecomposable, if there is an $H \in{ \Bbb N}$, such that
 the following holds for $K> H$:
For $p\in P(A(X))$ and $I \in{ \Bbb Z}$ such that 
$$
x_{(I, \infty)}= p_{(I, \infty)},
$$
and 
$$
y \in  A^+(X)\setminus A(X) , 
$$
such that
$$
y_{(-\infty, I+K]}= p_{(-\infty, I+K]},
$$
one has that
$$
(x _{(-\infty, I]}   , y_{(I, \infty)}  )\in X,
$$
implies that
$$
(x _{(-\infty, I]}   , y_{(I, \infty)})   \in A^+(X).
$$
The indecomposable points in $A^+(X)\cap Y(X)$ are defined symmetrically.

With the notion of an indecomposable point we  translate Condition (AQ5) into a Condition (DQ5):

\bigskip
\noindent
$\bold{(DQ5 - )}$ There exists an $M \in \Bbb N$ such that for $q, r \in P(A(X))$ and  an indecomposable point $x\in A^-(X) \cap  Y (X)$ and for $I_-, I_+ \in {\Bbb Z}, I_-< I_+$, such that
$$
x_{(-\infty,  I_-]}   = q_{-\infty,  0]} , \quad  x_{(I_+, \infty)} = r_{(0, \infty)},
$$
there exists an indecomposable pont $y  \in A^-(X) \cap Y (X)$ and $J_-, J_+ \in{ \Bbb Z}, J_-< J_+$, and $D \in \Bbb N$, such that
$$
J_+-J_-\leq M,
$$
$$
y_{(-\infty,  J_-]}   = q_{-\infty,  0]} , \quad  y_{(J_+, \infty)} = r_{(0, \infty)},
$$
and
$$
\Gamma( x_{( I_-   -D,  I_+ +D]}   ) = \Gamma(y_{( J_- -D,  J_+ +D]}).
$$

\bigskip
We have a Condition (DQ6) that also consists of two parts (DQ6$-$) and (DQ6$+$), that are symmetric to one another:

\bigskip
\noindent
$\bold{(DQ6 - )}$ If $x^{-}\in X_{(-\infty, 0 ]}$ is such that for all $M\in \Bbb N$  there is an $I> M$ such that
$$
x^{-}_{- I} \notin \omega^{-}_{1}(x^{-}_{(-I ,0]}),
$$
then for $y^{-}_{(-\infty, M]} \in X_{(-\infty, J} , J \in{ \Bbb N}$, such that 
$$
y^{-}_{(-\infty, 0]} =x^{-}_{(-\infty, 0]},
$$
and for all  $M\in \Bbb N$  there is an $I> M$ such that
$$
y^{-}_{- I} \notin \omega^{-}_{1}(y^{-}_{(-I ,J]}). 
$$

\bigskip

Condition (DQ6$-$) appeared in connection with the Cantor horizon of the Dyck shift \cite {KM1}.
Inspection shows that conditions (DQ1$-$) and (DQ2$-$), and therefore also conditions (DQ1$+$) and (DQ2$+$),  are invariant under topological conjugacy. Also, a topological conjugacy maps indecomposable points to indecomposable points. As a consequence, Condition  (DQ5$-$), and therefore also Condition (DQ5$+$),  is invariant under topological conjuagacy. We prove that  Condition (DQ6$-$) is invariant under topological conjugacy.

\begin{proposition}
Let $X\subset \Sigma^{{\Bbb Z}}, \widetilde {X}\subset \widetilde{\Sigma}^{{\Bbb Z}},$ be subshifts an let 
$\varphi: X \to  \widetilde{ X} $ be a topological conjugacy.
Let $X$ satisfy Condition  (DQ6$-$). Then $\widetilde {X}$ also satisfies Condition (DQ6$-$).
\end{proposition}
\begin{proof}
Let $L\in {\Bbb Z}_+$ be such that $[-L, L]$ is a coding window for $\varphi$ and $\varphi^{-1}$, and let $\varphi^{-1}$ be given by the block map $\Phi: \widetilde {X}_{[-L. L]}
\to \Sigma$. 

Let $\widetilde {x}^- \in \widetilde {X}_{(-\infty  ,  0]}$, and assume that for all $\widetilde{M}\in \Bbb N$ there is an $\widetilde{I}>\widetilde{M} $, such that
\begin{align*}
\widetilde {x}^-_{\widetilde{I}} \notin \omega^-_1(\widetilde {x}^-_{(-\widetilde{I}, 0]}) . 
 \tag {3.1}
\end{align*}
Set
$$
x^- = \Phi( \tilde x^-).
$$
For
$ MÊ\in \Bbb N$, setting $\widetilde{M}= M + L$, one has as a consequence of  (3.1), that there is an $\widetilde{I} > M + L$, such that
$$
x_{[\widetilde{I}-L, \widetilde{I}+L]}\notin \omega^-_{2L+1}(x_{( \widetilde{I}+L ,-L]}   ),
$$
and it follows that there is an 
$$
I\in [\widetilde{I}-L,\widetilde{I}+L],
$$
such that
$$
x_{-I} \notin \omega^-_1(x_{( -I , -L  ]}).
$$

To confirm $(DQ6-)$ for $ \widetilde {X}$, let
$$
 \quad \widetilde{ y}^- \in \widetilde {X}_{(-\infty, J]}, \quad J \in {\Bbb N}, 
$$
such that 
$$
 \widetilde {y}^- _{(\infty, 0]} =  \widetilde{ x}^-,
$$
and let $ \widetilde {M}\in \Bbb N$.
Set 
$$
y^- = \Phi (  \widetilde {y}^- ).
$$
As we have seen at the beginning of the proof,  we can apply  $(DQ6-)$ of $X$ to $x^-$, and we find that there is an
\begin{align*}
I >  \widetilde {M} + L,  \tag {3.2}
\end{align*}
such that
$$
y^-_{-I}\notin \omega_1^-(y^-_{(-I, J-L]}).
$$
This implies that
$$
\widetilde {y}^-_{(-I-L,-I+L]}\notin \omega_{2L+1}^-(\widetilde{ y}^-_{(-I+L, J]}).
$$
It follows that there exists an 
\begin{align*}
\widetilde{ I} \in[-I - L, -I + L], \tag {3.3}
\end{align*}
 such that
\begin{align*}
\widetilde {y}^-_{\widetilde {I}} \notin \omega_{1}^-
( \widetilde{ y}^-_{(-\widetilde{ I}, J  ]}). \tag {3.4}
\end{align*}
By (3.2), (3.3) and (3.4)  (DQ6$-$) is satisfied by $\widetilde { X}$.
\end{proof}

We recall from \cite {Kr2} the definition of property (A).
For $n \in \Bbb N$ a subshift $X \subset  \Sigma ^{{\Bbb 
Z}}$ such that $A_{1}(X) \neq \emptyset$,  
has property $(a, n, H),H \in{\Bbb N},$  if for 
$h,\widetilde {h} \ge 3H$ and for
$$
a \in A_{n}(X)_{[1,  h ]}, \quad \widetilde{a} \in A_{n}(X)_{[1,
 \widetilde{h} ]}, 
$$
such that
$$
a_{[1,  H]} = \widetilde{a}_{[1,  H]},\quad
a_{( h - H,  h ]} = \widetilde{a}_{(\widetilde{h} - H, \widetilde{h}]},
$$
one has that $a$ and $\tilde{a}$ have the same context. A subshift $X\subset  \Sigma ^{{\Bbb Z}}$
has Property $(A)$ if there are $H_{n}, n \in {\Bbb N}$, such
that $X$ has the properties $(a, n, H_{n}), n \in {\Bbb N} $.

We also recall the construction of the associated semigroup. 
Let $y, \widetilde{y}\in Y(X),$ let $y$ be left asymptotic to $q \in P(A(X))$ and right 
asymptotic to $r \in P(A(X)),$ and let $\widetilde {y}$ be left asymptotic to $ \widetilde{q} \in P(A(X))$ and right
asymptotic to $ \widetilde{r} \in P(A(X))$. Given that $X$ has the properties $(a, n, H_{n}), n \in \Bbb N,$ we say
that $y$ and $\widetilde {y}$ are equivalent, $y \approx \widetilde {y}$, if $q \approx
\widetilde{q}$ and $ r \approx \widetilde{r}$, and  
if for $n \in \Bbb N$ such that $ q, r, \widetilde {q}, \widetilde {r} 
\in A_{n}(X)$ and  for $I, J, \tilde
{I}, \widetilde {J} \in \Bbb Z,  I < J, \widetilde {I}< \widetilde{J},$ such that
$$
y_{(- \infty, I]} = q_{(- \infty, 0]},\quad  y_{(  J,\infty)} = r_{(  0,\infty)},
$$
$$
\widetilde{y}_{(- \infty, \widetilde{I}]} = \widetilde{q}_{(- \infty, 0]},\quad \widetilde{ y}_{(  \widetilde{J},\infty)} = \widetilde{r}_{(  0,\infty)},
$$
one has for $h\geq 3 H_{n}$ and for 
$$ 
a \in X_{(I  -  h,J + h]}, \quad
\widetilde {a} \in X_{(\widetilde {I} -  h,\widetilde {J} + h]}, 
$$
such that  
$$
a _{(I-  H_{n} ,  J + H_{n} ]} = y_{(I - H_{n} ,  J + H_{n} ]} ,\quad
 \widetilde {a} _{(\widetilde {I} -  H_{n} ,  \widetilde {J} + H_{n} ]} =
\tilde{y}_{(\tilde {I} -  H_{n} ,  \tilde {J} + H_{n} ]},
$$
$$
a _{(I - h ,  I  - h + H_{n})} = \widetilde {a} _{(\widetilde {I} -
h , \widetilde {I}  - h+ H_{n} )},
$$
$$
a _{(J + h- H_{n} ,  J  + h ]} =
 \widetilde {a} _{(\widetilde {J }+ h - H_{n}, 
\widetilde { J}  + h ]},
$$
and such that
$$
a_{(I - h, I]} \in A_{n}(X)_{(I - h, I]} , \quad
\widetilde{a}_{( \widetilde{I}- h, \widetilde{I}]}\in A_{n}(X)_{( \widetilde{I}- h, \widetilde{I}]}\ ,
$$
$$
 a_{(J , J+h]} \in A_{n}(X)_{(J , J+h]}  , \quad
  \widetilde{a}_{( \widetilde{J}  , \widetilde{J}+h]}  \in A_{n}(X)_{( \widetilde{J}  , \widetilde{J}+h]} ,
$$
that $a$ and $\widetilde {a}$ have the same context.
To give $[Y(X)]_{\approx}$ the
structure of a semigroup ${\mathcal S}(X)$, that is invariantly associated to $X$, 
let $u, v \in Y(X)$, let $u$ be right asymptotic to
$q \in P(A(X))$ and let  $v$ be left asymptotic to $r \in P(A(X))$. 
If here 
$q \succeq r $, then set $[u]_{\approx}[v]_{\approx}$  equal to $[y]_{\approx}$,
where $y$ is any point in $Y(X)$ such that there are
$n \in \Bbb N, I, J ,\widehat {I},  \widehat {J} \in \Bbb Z , I < J, \widehat 
{I}<   \widehat {J}, $ such that $ q, r \in A_{n}(X),$  and such that
$$
u_{(I, \infty)} = q_{(I, \infty)}, \quad
v_{(-\infty, J]} = r_{(-\infty, J]},
$$  
$$
y_{(- \infty,\widehat{I} + H_{n}]}= u_{(- \infty,I + H_{n}]} ,\quad
y_{ (\widehat{J}  - H_{n}, \infty)}= v_{ (J - H_{n}, \infty)},
$$
and 
$$
y_{(\widehat{I}  , \widehat{J}  ]}\in A_{n}(X)_{(\widehat{I}  , \widehat{J}  ]},
$$
provided that such a point $y$ exists. If such a point $y$ does not exist,  set
$[u]_{\approx}[v]_{\approx}$  equal to zero.
Also, in the case that 
$q\not \succeq r, $ set $ [u]_{\approx}[v]_{\approx}$ 
 equal to zero.

For a subshift with  Property (A) conditions (AP1), (AP2), (AQ1), (AQ 2) and (AQ4)  are equivalent to their translations.
We prove that for a subshift with  Property (A)  conditions (DQ4) and  (DQ6) together imply  Condition (AQ6).

\begin{proposition}
Let the subshift $X$ have property (A) and let conditions (DQ3),  (DQ4) and (DQ6) be satisfied by $X$.
Then ${\mathcal S}(X)$ satisfies Condition (AQ6).
\end{proposition}
\begin{proof}
If ${\mathcal S}(X)$ does not satisfiy  (AQ6$-$) then there exist
$$
H^-_m, G^-_m\in{\mathcal S}^-(X)\setminus \mathcal U_{\mathcal S(X)}, \qquad m \in{ \Bbb Z_+},
$$
such that
\begin{align*}
H^-_m = H^-_{m+1}G^-_m, \quad m \in{ \Bbb Z_+}, \tag {3.5}
\end{align*}
or such that
\begin{align*}
H^-_m= G^-_mH^-_{m+1}, \qquad m \in {\Bbb Z_+}. \tag {3.6}
\end{align*}
Assume (3.5). Let $U\in {\mathcal U}_{\mathcal S}$
be given by 
\begin{align*}
UH_0^-= H_0^-,
\end{align*}
and let  
$U_m\in {\mathcal U}_{\mathcal S}, m \in{ \Bbb Z}_+,$
be given by
$$
U_{m+1}G^-_m = G^-_m, \qquad m \in{ \Bbb Z}_+.
$$
By (DQ3$-$) there exists an $H^+ \in {\mathcal S}^+(X)(U, U_{0})
\setminus \mathcal U_{\mathcal S(X)},$ such that
$$
H^-_0H^+ = U,
$$
and  also 
$
G^+_m \in {\mathcal S}^+(X)(U_{m+1}, U_{m}) \setminus \mathcal U_{\mathcal S(X)}, 
$
such that
\begin{align*}
G^-_mG^+_m=U_m,\quad m \in{ \Bbb Z}_+. \tag{3.7}
\end{align*}
One has also $G_m \in \mathcal S(X) \setminus \mathcal U_{\mathcal S(X)},   m \in{ \Bbb Z}_+,$ such that
\begin{align*}
 U G_m = G_m,       \quad  G^-_mG_m=     0,      \qquad  m \in{ \Bbb Z}_+,          \tag {3.8}
\end{align*}
for otherwise one would have for some $m_\circ \in \Bbb N, $ that $G_{m_\circ} \in   \mathcal U_{\mathcal S(X)}$.
Let  
$V_m\in {\mathcal U}_{\mathcal S(X)}, m \in{ \Bbb Z}_+,$
be given by
$$
G_mV_m = G_m, \qquad m \in{ \Bbb Z}_+.
$$
Let $\Delta$ be a system of representatives of $\approx$-equivalence in $P(A(X))$, let $n\in\Bbb N$ be such that every point in $\Delta$ has a period that is less or equal to $n$, let  
$H_n\in\Bbb N $ be such that $X$ has Property $(a,n, H_n)$, and set $h = 3 H_n.$ In view of
$$
 \mathcal U_{\mathcal S(X)} = \frak P(X),
$$
we can choose a $ p \in \Delta$, such that
$$
U = [p]_\approx,  
$$
and also $p^{(m)},  q^{(m)}\in \Delta, m \in{ \Bbb Z}_+$, such that
$$
U_m = [p^{(m)}]_\approx, \quad V_m = [ q^{(m)}]_\approx,\qquad  m \in{ \Bbb Z}_+,
$$
together with an $J> n$, and  
 $J_m^-,J_m^+, K_m \in \Bbb N, m \in \Bbb N, $
such that one has, setting
$$
J_0^-=J_0^+ =0,
$$
that
$$
J_m^->J_{m-1}^-+ h, \qquad J_m^+>J_{m-1}^++ h, \qquad m \in \Bbb N,
$$
and also
$$
K_m>J_{m-1}^++ h,  \qquad m \in \Bbb N,
$$
together with
$
 a\in X_{( 0,J+ h]}
$
such that
\begin{align*}
a_{( 0 ,   h ]} = p^{(0)}_{(0 ,   h  ]},
\quad
a_{( J ,J+   h ]} = p_{(0 ,   h  ]}, \tag {3.9}
\end{align*}
\begin{align*}
[(p^{(0)}_{(-\infty, h]}, a_{(  h,  J ]}, p_{(0, \infty}  )]_\approx = H_0^+, 
\end{align*}
and 
$
x^- \in X_{( - \infty, h]}
$
such that
\begin{align*}
&x^-_{( J_m^-   -h ,   J_m^-  ]} = p^{(m)}_{(-h ,   0  ]}, \qquad m \in \Bbb Z_+, \tag {3.10} \\
&[(p^{(m)}_{(-\infty, h]}, x^-_{( -J^-_m  + h_, - J_{m-1}]}, p^{(m-1}_{(0, \infty}  )]_\approx = G^-_m,  \qquad  m \in \Bbb N,
\end{align*}
and
$
a^{(m)} \in X_{(0, K_m + h]}, m \in \Bbb N,
$
such that
\begin{align*}
&a^{(m)}_{( J_\ell^+  ,    J_\ell^+ +h ]} = p^{(\ell)}_{(0 ,h ]}, \tag {3.11}
\\
&[(p^{(\ell)}_{(-\infty, h]}, a^{(m)}_{(J^+_\ell  + h,  J^+_{\ell + 1} ]}, p^{(\ell +1)}_{(0, \infty}  )]_\approx = G^+_\ell, \qquad 0 \leq \ell < m,
\\
&a^{(m)}_{( K_m   ,   K_m +h  ]} = q^{(m)}_{(0 ,   h  ]}, \tag {3.12}
\\
&[(p^{(m)}_{(-\infty, h]}, a^{(m)}_{( J^+_m  + h_, K_m ]}, q^{(m)}_{(0, \infty}  )]_\approx = G_m,  \qquad  m \in \Bbb N.
\end{align*}
By  (3.7), (3.10) and (3.11), and by the choice of $h$, 
$$
(x^-_{(-J_{m-1}^-,  0]}, a^{(m)}) \in X_[{-J_{m-1}^-,  K_{m}+h] },
$$
and by (3.8), (3.10) and (3.12), and by the choice of $h$, 
$$
(x^-_{(-J_m^-,  0]}, a^{(m)}) \not\in X_{(-J_m^-,  K_{m}+h] }.
$$
It follows that there is an $i \in (  J_{m-1}^-  ,   -J_m^-    ]$
such   that
$
x^-_{-i} \notin \omega_1^-(x^-_{(-i, 0]}).
$
On the other hand it follows from
$$
H^-_m(\prod_{m \leq \ell \leq 0} G_\ell^-)H^+ = U,
$$
that
$$
(\prod_{m \leq \ell \leq 0} G_\ell^-)H^+  \in \mathcal S^+(X),
$$
and by (3.9) and (3.10) then
$$
x_{-i} \in \omega_1^-(x_{(-i, 0]}, a  ), \qquad i \in \Bbb N,
$$
and we have a contradiction to (DQ6-).
The case (3.6) reduces by  (DQ4$+$) 
to the  case (3.5).
The proof, that (DQ3+),(DQ4-), and (DQ6+) together imply (DQ6-),  is symmetric.
\end{proof}

\begin{corollary}
A subshift with property (A), that satisfies  conditions (DP1$-$2) and conditions (DQ1$-$6) has as its associated semigroup an ${\mathcal R}$-graph semigroup.
\end{corollary}
\begin{proof}
By Theorem  2.3 of   \cite{HK2} for  a semigroup $\mathcal  S$ that is associated to a subshift with Property $(A)$ the projection of $\mathcal  S$ onto $[\mathcal  S]$ is an isomorphism.
Taking into account that  (DP1) and (DP2) are the translations of   (AP1) and (AP2) and that  (DQ3), (DQ4) and (D5) are the translations of  (AQ3), (AQ4) and (AQ5), apply Proposition 3.2 and Theorem 2.1.
\end{proof}

\section{Property $(B)$}
We say that a subshift $X\subset \Sigma^{\Bbb Z}$ has property $(B)$ with respect to the parameter $M \in \Bbb N$ if the following holds: For $I_{-}, I _{+}, J_{-}  ,  J_{+}  \in \Bbb Z$ such that
$$
I _{+}  - I_{-}  ,  J_{+}  - J_{-}    \geq M,
$$
and
$$
a \in X_{(I_{-}, I _{+}]}, \quad b \in X_{(J_{-},J_{+}]}, 
$$
such that 
$$
a_{(I_{-}, I_{-} +M]} = b_{(J_{-}, J_{-} +M]} , \qquad
a_{(I_{+}-M, I_{+} ]} = b_{(J_{+}-M, J_{+}]} ,
$$
and 
$R \in \Bbb N$ and
$$
x^{-} \in \omega^-_\infty (a)\cap \omega^-_\infty (b),\quad x^{+}\in 
\omega^+_\infty (a)\cap\omega^+_\infty (b),
$$ 
such that
$$
\Gamma (x^{-}_{(I_{-}- R, I_{-} ]}, a, x^{+}_{(I_{+}, I_{+}+R ]}  ) =
\Gamma (x^{-}_{(J_{-}- R, J_{-} ]}, b, x^{+}_{(J_{+}, J_{+}+R ]}  ) ,
$$
one has that
$$
\Gamma (a) = \Gamma (b).
$$
We say that a subshift $X\subset \Sigma^{\Bbb Z}$ has property  $(B)$ if it has property  $(B)$ with respect to some parameter.

For a graph inverse semigroup $\mathcal S(\frak P, \mathcal E^-, \mathcal E^- )$ an $\mathcal S(\frak P, \mathcal E^-, \mathcal E^- )$-presentation has Property $(B)$, as can be seen by inspection. Conversely,  if an $\mathcal S_\mathcal R(\frak P, \mathcal E^-, \mathcal E^- )$ has properties $(A)$ and $(B)$, then the semigroup $\mathcal S_\mathcal R(\frak P, \mathcal E^-, \mathcal E^- )$ is necessarily a graph inverse semigroup.

\begin{theorem}
Property $(B)$ is an invariant of topological conjugacy.
\end{theorem}
\begin{proof}
If a subshift 
$X\subset \Sigma^{\Bbb Z}$ has Property $B$ with respect to some parameter, then its higher block systems also have Property $B$.
To prove the proposition it is therefore enough to consider the situation that one is given subshifts
$X\subset \Sigma^{\Bbb Z}, \widetilde{X}\subset\widetilde{ \Sigma}^{\Bbb Z}$ and a topological conjugacy
$\varphi : X \to \widetilde{X}$
that is given by a one-block map $\Phi: \Sigma \to \widetilde {\Sigma}$
with $\varphi^{-1}$ given for some $L \in \Bbb Z_{+}$ by a bock map $\widetilde{\Phi}: \widetilde{X}_{[-L, L]} \to \Sigma$, where $X$ has Property $(B)$ with respect to the parameter $M$, and to prove that the subshift $\widetilde{X}$ has also Property $(B)$.
We set
$$
\widetilde {M} = M + 2L,
$$
and we prove that $\widetilde{X}$ has Property $(B)$  with respect to the parameter $\widetilde {M}$.
For this let
${I}_{-},{I} _{+}, {J}_{-}  , { J}_{+}  \in \Bbb Z,$
$$
{I} _{+}  - {I}_{-}  ,  {J}_{+}  -{ J}_{-}    \geq \widetilde{M},
$$
and let
$$
\widetilde{a} \in \widetilde{X}_{({I}_{-}, {I}_{+}]}, \quad \widetilde{b} \in \widetilde{X}_{({ J}_{-},{ J}_{+}]}, 
$$
be such that 
\begin{align*}
\widetilde{a}_{({I}_{-},{I}_{-} +\widetilde{M}]} = \widetilde{b}_{{J}_{-}, {J}_{-} +\widetilde{M}]} , \quad
\widetilde{a}_{({I}_{+}-\widetilde{M},{ I}_{+} ]} = \widetilde{b}_{({J}_{+}-\widetilde{M},{J}_{+}]} , \tag {4.1}
\end{align*}
and let $R \in \Bbb N$ and 
\begin{align*}
\widetilde{x}^{-} \in \omega^-_\infty (\widetilde{a})\cap \omega^-_\infty (\widetilde{b}),\quad \widetilde{x}^{+}\in \omega^+_\infty (\widetilde{a})\cap\omega^+_\infty (\widetilde{b}), \tag {4.2}
\end{align*}
 be such that
\begin{align*}
\Gamma (\widetilde{x}^{-}_{({I}_{-}- R,{ I}_{-} ]}, \widetilde{a}, \widetilde{x}^{+}_{({I}_{+}, {I}_{+}+R ]}  ) = \Gamma (\widetilde{x}^{-}_{({J}_{-}- R,{J}_{-} ]}, \widetilde{b}, \widetilde{x}^{+}_{({J}_{+}-, {J}_{+}+R ]}  ) . 
\tag {4.3}
\end{align*}
We have to prove that
\begin{align*}
\Gamma (\widetilde{a})= \Gamma (\widetilde{b}). \tag {4.4}
\end{align*}
We let
$$
a \in X_{(I_{-}+L, I_{+}-L]}, \quad b \in X_{(J_{-}+L, J_{+}-L]},
$$
be given by 
$$
a = \widetilde{\Phi} ( \widetilde{a}  ), \qquad b= \widetilde{\Phi} ( \widetilde{b}  ). 
$$
By (4.2) 
\begin{align*}
a_{(I_{-} +L ,  I_{-}+L +M ]} =    b_{(J_{-}+L  ,  J_{-}+L+M ]} , \quad 
a_{(I_{+} -L- M ,  I_{+}-L ]} =   b_{(J_{-} -L-M ,  J_{+}-L ]}. \tag {4.5}
\end{align*}
We set also
$$
x^{-} =  \widetilde{\Phi} ( \widetilde{x} ^{-},  \widetilde{a}_{( {I}_{-}  , {I}_{-}  + 2L ]}), \quad 
x^{+} =  \widetilde{\Phi} (  \widetilde{a}_{( {I}_{+} -2L , {I}_{+}   ],} \widetilde{x} ^{+}) .
$$
It follows from (4.1) and (4.2) that
\begin{align*}
x^{-}  \in \omega^-_\infty (a) \cap \omega^-_\infty (b), \quad  x^{+} \in \omega^+_\infty (a) \cap \omega^+_\infty (b).
\tag{4.6}
\end{align*}
One has
\begin{multline*}
\Gamma (x^{-}_{(  I_{-}-R-L  , I_{-}+L   ]}, a, x^{+}_{(  I_{+}-L  , I_{+}+R+L   ]})= \\
\{  ( \widetilde{\Phi} (  \widetilde{u}^-   ) ,    \widetilde{\Phi}(  \widetilde{u}^+    ) :
(  \widetilde{u}^-    ,   \widetilde{u}^+ ) \in\Gamma (\widetilde{x}^{-}_{({I}_{-}- R,{ I}_{-} ]}, \widetilde{a}, \widetilde{x}^{+}_{({I}_{+}, {I}_{+}+R ]}  )  \}
\end{multline*}
with a similar expression for
$
\Gamma (x^{-}_{(  J_{-}-R-L  , J_{-}+L   ]}, b, x^{+}_{(  J_{+}-L  , J_{+}+R+L   ]}),
$
and it is seen that (4.3) implies that
\begin{align*}
\Gamma (x^{-}_{(  I_{-}-R-L  , I_{-}+L   ]}, &a, x^{+}_{(  I_{+}-L  , I_{+}+R+L   ]})=   \tag {4.7} \\ 
&\Gamma (x^{-}_{(  J_{-}-R-L  , J_{-}+L   ]}, b, x^{+}_{(  J_{+}-L  , J_{+}+R+L   ]}).
\end{align*}
By  (4.5), (4.6) and  (4.7) we can apply Property $(B)$ of $X$ to obtain
\begin{align*}
\Gamma (a) = \Gamma (b)  . \tag {4.8}
\end{align*}
One has
\begin{multline*}
\Gamma (\widetilde{a}) = 
\{ ( {\Phi} ( {u}^- _{(-\infty, I_-]}  ) ,   {\Phi}(  {u}^+ _ {(I_+, \infty)}  ) ): (  {u}^-  , {u}^+   ) \in \Gamma(a), \\
{\Phi} ( {u}^-_{( I_-  , I_-+L  ]}   ) = 
  \widetilde{a}_{( I_-  , I_-+L  ]} ,
  {\Phi} ( {u}^+_{( I_{+}  , I_{+}+L  ]}   ) = 
  \widetilde{a}_{( I_{+}  , I_{+}+L  ]}  \}
\end{multline*}
with a similar expression for $\Gamma (\widetilde{b})$, from which it is seen that (4.8) implies (4.4).
\end{proof}

\begin{lemma}
Let $X\subset \Sigma^{\Bbb Z},
\widetilde{X}\subset\widetilde{ \Sigma}^{\Bbb Z}, $ be subshifts, and let 
 ${\varphi} :  {X}\to  \widetilde X $ 
 be a topological conjugacy that is given by a one-block map 
 ${\Phi}: \Sigma \to \widetilde \Sigma$
with ${\varphi}^{-1}$ given for some $L \in \Bbb Z_{+}$ by a bock map $\widetilde\Phi: 
\widetilde{X}_{[-L, L]} \to{ \Sigma}$. Let $\widetilde{X}$ have property $(B)$ with respect to the parameter $L$. Let $I_{-}, I _{+}, J_{-}  ,  J_{+}  \in \Bbb Z$ be such that
$$
I _{+}  - I_{-}  ,  J_{+} - J_{-}    \geq L,
$$
and let
$$
\widetilde{a} \in  \widetilde{X}_{( I_{-}  , I _{+}  ]}, \quad \widetilde{b} \in \widetilde{X}_{( J_{-}  , J_{+}  ]}, 
$$
be such that
$$
\widetilde{a}_{(I_{-}, I_{-} +L]} = \widetilde{b}_{(J_{-}, J_{-} +L]} , \quad
\widetilde{a}_{(I_{+}-L, I_{+} ]} =\widetilde{ b}_{(J_{+}-L, J_{+}]} , 
$$
and let
\begin{align*}
\widetilde{x}^{-} \in \omega^-_\infty (\widetilde{a})\cap \omega^-_\infty (\widetilde{b}),\quad \widetilde{x}^{+}\in \omega^+_\infty (\widetilde{a})\cap\omega^+_\infty (\widetilde{b}), 
\end{align*} 
Set
$$
{a} = \widetilde{\Phi} (\widetilde{ x}^{-} _{({I}_{-}-L,{I}_{-}]}, \widetilde{a},  \widetilde{x}^{+} _{({I}_{+},{I}_{+}+L]}  )  ,  
 \quad {b} =\widetilde { \Phi }( \widetilde{x}^{-}_{({J}_{-}-L,{J}_{-}]}, \widetilde {b}, \widetilde{ x}^{+} _{({J}_{+},{J}_{+}+L]}  ).
$$
Then
\begin{align*}
\Gamma({a}) = \Gamma({b}) , \tag {4.9}
\end{align*}
implies
\begin{align*}
\Gamma(\widetilde{a}) = \Gamma(\widetilde{b}). \tag {4.10}
\end{align*}
\end{lemma}
\begin{proof}
 It is
\begin{align*}
\Gamma  (\widetilde{ x}^{-}& _{({I}_{-}-L,{I}_{-}]}, \widetilde{a},  \widetilde{x}^{+} _{({I}_{+},{I}_{+}+L]}  )  )  =\\
\{  (&\Phi( {u}^{-}_{(-\infty, {I}_{-}-L] } ), \Phi ( {u}^{+}_{({I}_{+}+L ,\infty)}     )):   ( {u}^{-} , {u}^{+}  ) \in \Gamma (a), \\
     &\Phi( {u}^{-}_{( {I}_{-}-L,{I}_{-}] }) =  \widetilde {x}^{-}_{( {I}_{-}-L,{I}_{-}] }      ,
                    \Phi( {u}^{+}_{( {I}_{+},{I}_{+}+L] } )=  
                         \widetilde {x}^{+}_{( {I}_{+},{I}_{+}+L] }  \},
\end{align*}
and replacing here in the right hand side  $a$ by $b$ yields the corresponding expression for $ \Gamma( \widetilde{x}^{-}_{({J}_{-}-L,{J}_{-}]}, \widetilde {b}, \widetilde{ x}^{+} _{({J}_{+},{J}_{+}+L]}  )$. From this it is seen that (4.9) implies
$$
\Gamma  (\widetilde{ x}^{-} _{({I}_{-}-L,{I}_{-}]}, \widetilde{a},  \widetilde{x}^{+} _{({I}_{+},{I}_{+}+L]}  )  )  =
\Gamma  (\widetilde{ x}^{-} _{({I}_{-}-L,{I}_{-}]}, \widetilde{b},  \widetilde{x}^{+} _{({I}_{+},{I}_{+}+L]}  )  )  ,
$$
which then by Property $(B)$ of $\widetilde{ X}$ implies (4.10).
\end{proof}

\section{Property $(c)$}
We say that a subshift $X \subset \Sigma ^{\Bbb Z}$  has property $(c)$ with respect to the parameter $Q\in \Bbb N$
if the following holds:
For $I_{-} ,  I_{+}  , J_{-}   , J_{+}  \in \Bbb Z_{+},$
$$
I_{+}   -  I_{-}  ,  J_{+} -  J_{-}  \geq Q,
$$
and for
$$
a\in X_{( I_{-} , I_{+} ]},   \quad b\in X_{( J_{-} , J_{+} ]},
$$
such that
$$
\omega^-_\infty (a) \cap \omega^-_\infty  (b)  \neq  \emptyset, \quad \omega^+_\infty (a) \cap \omega^+_\infty (b)  \neq  \emptyset,
$$
and such that
$$
a_{( I_{-} , I_{-}+Q ]} = b_{( J_{-} , J_{-} +Q]},
\qquad
a_{( I_{+}-Q, I_{+} ]} = b_{( J_{+} -Q, J_{+} ]},
$$
and
$$
\Gamma^{-}(a)= \Gamma^{-}(b), \qquad\Gamma^{+}(a) = \Gamma^{+}(b ),
$$
one has that
 $$
\Gamma (a)= \Gamma (b) .
$$

We say that a subshift $X\subset \Sigma^{\Bbb Z}$ has Property $(c)$ if  it has Property $(c)$ with respect to some parameter.
Inspection shows that $S_{\mathcal R}({\frak P}, {\mathcal E}^-,  {\mathcal E}^+ )$-presentations have Property $(c)$.  

\begin{proposition}
The simultaneous occurrence of properties $(B)$ and $(c)$ is
an invariant of topological conjugacy.
\end{proposition}
\begin{proof}
By Proposition 5.3 Property $(B)$ is an invariant of topological conjugacy. Also, if a subshift 
$X\subset \Sigma^{\Bbb Z}$ has Property $(c)$, then its higher block systems also have Property $(c)$. To prove the proposition it is therefore enough to consider the situation that one is given subshifts
$X\subset \Sigma^{\Bbb Z}, \widetilde{X}\subset\widetilde{ \Sigma}^{\Bbb Z}$ and a topological conjugacy
$\varphi : X \to \widetilde{X}$
that is given by a one-block map $\Phi: \Sigma \to\widetilde \Sigma$
with $\varphi^{-1}$ given for some $L \in \Bbb Z_{+}$ by a bock map $\widetilde{\Phi}: 
\widetilde{X}_{[-L, L]} \to \Sigma$, where $X$ has Property $(B)$ with respect to the parameter $L$ and has Property $(c)$ with respect to the parameter
$Q $, and to prove that the subshift $\widetilde{X}\subset\widetilde{ \Sigma}^{\Bbb Z}$ has Property $(c)$.
We set
$$
\widetilde{Q} = Q+L.
$$
We prove that $\widetilde{X}$ has Property $(c)$ with respect to the parameter $ \widetilde{Q} $.
For this, let 
$I_{-} ,  I_{+}  , J_{-}   , J_{+}  \in \Bbb Z_{+},$ be such that
$$
I_{+}   -  I_{-}  ,  J_{+} -  J_{-}  \geq  \widetilde{Q},
$$
and let
$$
 \widetilde{a}\in \widetilde{X}_{( I_{-} , I_{+} ]},   \quad  \widetilde{b}\in  \widetilde{X}_{( J_{-} , J_{+} ]},
$$
be such that
\begin{align*}
\omega^-_\infty (\widetilde{a}) \cap \omega^-_\infty   ( \widetilde{b})  \neq  \emptyset, \quad \omega^+_\infty  (\widetilde{a}) \cap \omega^+_\infty  ( \widetilde{b})  \neq  \emptyset, \tag {5.1}
\end{align*}
and such that
\begin{align*}
\widetilde{a}_{( I_{-} , I_{-}+\widetilde{Q} ]} = 
\widetilde{b}_{( J_{-} , J_{-}+\widetilde{Q}]},\qquad \widetilde{a}_{( I_{+} -\widetilde{Q}, I_{+} ]} = 
 \widetilde{b}_{( J_{+} -\widetilde{Q}, J_{+} ]}, \tag {5.2}
\end{align*}
\begin{align*}
\Gamma^{-}(\widetilde{a}) = \Gamma^{-}(\widetilde{b} ), \tag {5.3}
\end{align*}
and
$$
\Gamma^{+}(\widetilde{a})=\Gamma^{+}( \widetilde{b} ).
$$
We have to prove that
\begin{align*}
\Gamma ( \widetilde{a})= \Gamma (\widetilde{b}) .  \tag {5.4}
\end{align*}
By (5.1) we can choose
\begin{align*}
 \widetilde{x}^{-}   \in \omega^-_\infty (\widetilde a) \cap \omega^-_\infty  (\widetilde b)  ,  \tag {5.5}
 \end{align*}
 and
 \begin{align*}
 \widetilde{x}^{+}   \in \omega^+ _\infty(\widetilde a) \cap \omega^+_\infty (\widetilde b) . \tag {5.6}
\end{align*}
We set
$$
{a} = \widetilde{\Phi} (\widetilde{ x}^{-} _{({I}_{-}-L,{I}_{-}]}, \widetilde{a},  \widetilde{x}^{+} _{({I}_{+},{I}_{+}+L]}  )  ,  
 \quad {b} =\widetilde { \Phi }( \widetilde{x}^{-}_{({J}_{-}-L,{J}_{-}]}, \widetilde {b}, \widetilde{ x}^{+} _{({J}_{+},{J}_{+}+L]}  ).
$$
By Lemma 4.2 and by  Property $(B)$ of $ \widetilde{X}$ we will have proved (5.4), once we have shown  that 
\begin{align*}
\Gamma (a) = \Gamma (b). \tag {5.7}
\end{align*}
We set
$$
x^- = \widetilde\Phi( \widetilde{ x}^{-} , \widetilde a_{( {I}_{-},{I}_{-}+L]}    ) ,   
 \quad x^+ = \widetilde\Phi (  \widetilde a_{( {I}_{+}-L,{I}_{+}]}, \widetilde{ x}^{-}   ).
$$
By (5.2)  then also
\begin{align*}
x^- = \widetilde\Phi( \widetilde{ x}^{-} , \widetilde b_{( {J}_{-},{J}_{-}+L]}    ) ,   
\quad x^+ = \widetilde\Phi (  \widetilde b_{( {J}_{+}-L,{J}_{+}]}, \widetilde{ x}^{-}   ), 
\end{align*}
and by (5.5) and (5.6)  then
\begin{align*}
x^- \in \omega^-_\infty (a) \cap \omega^-_\infty  (b), \qquad x^+\in \omega^+_\infty (a) \cap \omega^+_\infty (b) . \tag {5.8}
\end{align*}
By  (5.2) 
\begin{align*}
a_{( I_{-} , I_{-} +Q]} = &\widetilde {\Phi} (\widetilde{x}^{-}_{( I_{-}-L, I_{-}] }   , \widetilde{a}_{( I_{-} , I_{-} +\widetilde{Q} ]} )=     \tag {5.9}\\
& \widetilde {\Phi}( \widetilde{x}^{-}_{( I_{-}-L, I_{-}] }  ,  \widetilde{b}_{( J_{+}  -\widetilde{Q}, J_{+} ]} )= b_{( J_{-} , J_{-} +Q ]},
 \end{align*}
 and symmetrically
 \begin{align*}
a_{( I_{-} , I_{-} +Q ]} = b_{( J_{-} , J_{-} +Q ]}. \tag {5.10}
 \end{align*} 
We prove that
\begin{align*}
\Gamma^-(a) =\Gamma^-(b). \tag {5.11}
\end{align*}
For this let $u^-\in \Gamma^{-}(a)$ and set $\widetilde u = \Phi ( u^-)$. One has $\widetilde u \in \Gamma^{-}(\widetilde{a})$, which by (5.3) implies that $\widetilde u \in  \Gamma^{-}(\widetilde{b}).$  
By (5.6) then $(\widetilde u, \widetilde b , \widetilde{ x}^{+})\in \widetilde X$, and therefore
$
( u,  b ,{ x}^{+})\in X
$
and (5.10) is confirmed.
Symmetrically one has that
\begin{align*}
\Gamma^+(a) =\Gamma^+(b). \tag {5.12}
\end{align*}
We use (5.8-12) and Property $(c)$ of $X$ to show (5.7).
\end{proof}

\section{Instantaneous presentations}
We say that a subshift $X\subset \Sigma^{\Bbb Z}$ is right instantaneous 
if $ \omega^+_{1} ( \sigma )\neq \emptyset ,\sigma \in \Sigma $, equivalently, if  $ \omega^+_{\infty} ( \sigma )\neq \emptyset, \sigma \in \Sigma $. Left instantaneity is defined time 
symmetrically.  The notion of left instantaneity was  
considered by 
Matsumoto in \cite [ Section 4] {M1}. We say that a subshift is bi-instantaneous, if it is 
left and right instantaneous. We give an example of a topologically 
transitive sofic system that is left instantaneous but not 
right instantaneous. This example 
is a variation of an example of Carlsen and Matsumoto in \cite {CM}. 
Let $\Sigma = \{ 0,1, \alpha, \beta \}$ and exclude from 
$\Sigma^{{\Bbb Z}}$ the points that contain one 
of the following words: $10, 11,  \alpha 0^{n}1\beta,  \beta 0^{n}1\alpha ,  n \in  {\Bbb N}$. In 
this way one obtains a left instantaneous sofic system in which the words of the form 
$0^{n}1, \ n \in {\Bbb N}$, do not have a future that is compatible with their entire past.  By a product construction one can obtain examples of this 
type of topologically transitive sofic systems that are neither left 
nor right instantaneous.

\begin{theorem}
A sofic system has a bi-instantaneous presentation.
\end{theorem}
\begin{proof}
For the construction of a bi-instantaneous presentation of a sofic system 
$X
\subset \Sigma^{ \Bbb Z}$ set 
$$
 {\mathcal V}= \{ \Gamma^{+}_\infty(x^{-}): x^{-}\in X_{( - \infty, 0 ]} \},
$$
and denote for an admissible word $a$ of $X$  by $  {\mathcal V} (a)$ the
set of $V \in  {\mathcal V}  $ that contain a sequence that starts with $a$.  
A word $a\in \mathcal L(X)$ determines a partial mapping $\tau_a$ of $  {\mathcal V} $ into itself, that has ${\mathcal V} (a)$  as its domain of definition, and that is given by
$$
\tau _{a} (V) = \{ y^{+} \in X_{[1, \infty)}: (a,  y^{+} ) \in   V(a)\},\qquad  V \in  {\mathcal V} (a) .
$$
For $x \in X$ there are $J, K \in \Bbb N$ such that 
$$
\tau _{x_{[1,J]}} = \tau _{x_{[1,J+K]}},
$$
in which case
$$
\tau _{x_{(J,J+K]}} \restriction \tau _{x_{[1,J]}}(  {\mathcal V}(x_{[1,J]}) ) = \text{id}.
$$
For a point $x\in X$ we can therefore define an  $I^+(x) \in {\Bbb N}$  as the smallest $I\in \Bbb N$, such that there is an $i \in (1, I)$ such that 
\begin{align*}
\tau_{x_{(i, I]}} 
 \restriction  
\tau_{x_{[1,i]}}( {\mathcal V}_{x_{[1, i]}}) =\text{id}, \tag {6.1}
\end{align*}
and an $i^+(x)\in (1, I^+(x))$ as  the uniquely determined $i \in (1, I^+(x))$ such that (6.1) holds. 
One has
\begin{align*}
I^+(S_X(x)) +1 \geq I^+(x),  \qquad x \in X. \tag {6.2}
\end{align*}
To see this,  note that
$$
\tau_{x_1}({\mathcal V}({x_{[1,  i^+(S_X(x))  ]}}))    \subset    {\mathcal V}({x_{(1, i^+(S_X(x))  ]}}),
$$
and
$$
\tau _{x_{(    i^+(S_X(x)), I^+(S_X(x))]}} 
\restriction  
\tau _{x_{(1,   i^+(S_X(x))]}}
({\mathcal V}(x_{(1, i^+(S_X(x)) ]}) = \text{id},
$$
imply that
\begin{align*}
\tau_{x_{(i^+(S_X(x)) +1,  I^+(S_X(x))  +1]}} 
 \restriction  
\tau_{x_{[1,i^+(S_X(x)) +1]}}( {\mathcal V}_{x_{[1,i^+(S_X(x)) +1]}}) =\text{id},  
\end{align*}
which means that the contrary of (6.1) would contradict the definition of $I^+$.
Denote by $z^+$ the point in $X_{( I^+(x) , \infty)}$ that carries the right infinite concatenation of the word
$x_{x_{(  i^+ ,   I^+ ]}}.$
It follows from 
$
{\mathcal V}_{x_{[1, i^+(x)]}} =  {\mathcal V}_{x_{[1, I^+(x)]}} ,
$
that
\begin{align*}
z^+ \in \omega^+( x_{[1, I^+(x)]}).   \tag {6.3}
\end{align*}
With $I^-$ and $z^-$ defined time symmetrically  one has
that
\begin{align*}
I^-(S_X(x)) -1 \leq I^-(x),  \qquad x \in X, \tag {6.4}
\end{align*}
and
\begin{align*}
z^- \in \omega^-( x_{ [I^-(x), 0)}).   \tag {6.5}
\end{align*}
We set
$$
\Xi (x)=    ( x_{ [I^-(x), 0)}         ,0,   x_{[1, I^+(x)]}   ),  \qquad \quad x \in X,
$$
and with $M$ denoting a bound for \{  $\vert I^-(x)\vert, I^+(x) :x \in X\}$ we  can by 
(6.2) and (6.4) define an embedding $\xi$  of $X$ into 
$(\cup_{1 \leq m \leq M}({\mathcal L}_m(X) \times \Sigma\times {\mathcal L}_m(X)))^{{\Bbb Z}}$ by 
$$
 \xi (x)  = ( \Xi (S^{-i}x)  )_{i\in \Bbb Z} ,         \qquad \quad x \in X.
$$
Set
$$
\Delta =  \Xi (X), \quad  Y =  \xi (X).
$$
We prove that $Y$ is a bi-instantaneous presentation of $X$. For this let
$$
( a(-), \sigma, a(+)  ) \in \Delta,
$$
and let $x \in X$ be such that
$$
\Xi (x)=( a(-), \sigma, a(+)  ).
$$
By (6.3) 
$$
  \Xi (S^{-1}_X((x_{(-\infty,  I^+(x) ]},z^+(x))) \in \omega^+_1( a(-), \sigma, a(+)  ),
$$
which confirms the right instantaneity of $Y$.
The proof that $Y$ is left instantaneous is time symmetric and uses  (6.5).
\end{proof}

We give an example of a semi-synchronizing 
(see \cite {Kr2} \cite {Kr5}) right-instantaneous non-sofic system that is not left instantaneous, but has a 
left instantaneous presentation. For this, take as alphabet the set
$$
\Sigma =  \{\bold 1 , \alpha_{\lambda}, \alpha_{\rho}, \beta_{\lambda}, 
\beta_{\rho}, \gamma_{\lambda}, \gamma_{\rho} \} ,
$$
and view $\Sigma$ as a generating set of ${\mathcal D}_{3}$ with relations
$$
\alpha_{\lambda} \alpha_{\rho} = \beta_{\lambda} \beta_{\rho} = 
\gamma_{\lambda} \gamma{\rho} = \bold 1, \  \alpha_{\lambda}\beta_{\rho} = 
\beta_{\lambda}\alpha_{\rho} = \alpha_{\lambda} \gamma_{\rho} = 
\gamma_{\lambda}
\alpha_{\rho} = \beta_{\lambda} \gamma_{\rho} = 
\gamma_{\lambda}\beta_{\rho} = 0.
$$
Intersect he Motzkin shift
$$
M_{3} = \{ x\in \Sigma^{\Bbb Z}: \prod_{I_{-}\leq i < I_{+}}^{} x_{i}  \neq  0 , 
I_{-}, I_{+}\in \Bbb Z , I_{-}< I_{+} \}
$$
with the subshift of finite type that is obtained 
by 
excluding from $\Sigma^{\Bbb Z}$ all points that contain one of the words 
$\gamma_{\lambda} \gamma_{\lambda}, 
\alpha_{\lambda}  \alpha_{\lambda} \gamma_{\lambda}, 
\beta_{\lambda}\beta_{\lambda}\gamma_{\lambda}$ to obtain a subshift 
$X$. There is a topological 
conjugacy of $X$ onto a subshift  $\bar{X}$ that is given by a 
3-block mapping $\Phi$,
$$	
\Phi( \alpha_{\lambda} \beta_{\lambda} \gamma_{\lambda}) = \Phi
(\beta_{\lambda} \alpha_{\lambda} \gamma_{\lambda}) =  
\gamma_{\lambda}
$$
$$
\Phi ( \sigma \sigma' \sigma'') = \sigma', \quad\sigma\sigma'' 
\notin   \{ \alpha_{\lambda} \gamma_{\lambda} , 
\beta_{\lambda}\gamma_{\lambda}\} .
$$
Whereas the subshift $X$ is bi-instantaneous, the subshifts $\bar{X}^{\langle 
[0,n)\rangle}, \ n \in  \Bbb N$, are right-instantaneous but not 
left-instantaneous: The 
words  $\gamma_{\lambda}\gamma_{\lambda} \bold 1^{n} ,\ n \in  \Bbb N$, 
do not have a past that is compatible with their entire future context.

For a subshift $X 
\subset  \Sigma^{\Bbb Z}$, for $L\in\Bbb Z_{+}$, and for mappings
$$
\Psi^{(+)}: X_{[-L,L] } \to X_{[1,L+1] } 
$$
we formulate a condition $(RI)$, that comes in two parts:
\begin{align*}
(RIa): \Psi^{(+)} (a) \in \Gamma ^{+} (x^{-}, a_{[-L, 0]}),  \   \qquad a &\in X_{[-L,L]},x^{-}
\in \Gamma ^{-}(a).
\\
(RIb):\Psi ^{(+)}(b_{(-2L + \ell ,-L )} ,a_{[-L, \ell]}) = \Psi ^{(+)}
&(b_{(-2L + \ell ,-L )} ,
a_{[-L, 0]},\Psi ^{(+)}(a)_{[1. \ell]}), 
\\
&b \in X_{(-2L, L)}, a \in X_{[-L,L]}, 1 \leq \ell \leq L.
\end{align*}
We say that a mapping $ \Psi^{(+)} : X_{[-L,L]} \to X_{[1,L+1]}, L\in \Bbb Z_+,$ that satisfies
Condition  $(RI)$  is an $ RI$-mapping of $X$. 

\begin{lemma}
Let 
 $\widetilde{X} \subset \widetilde{\Sigma} ^{\Bbb Z}$ be a right instantaneous subshift, and let
 $\varphi $ be a topological conjugacy of $\widetilde{X}$ onto a subshift $X\subset \Sigma^{\Bbb Z}$, 
 that is given by a one-block map 
$\widetilde\Phi : \widetilde{\Sigma}\to \Sigma$
with  $\varphi^{-1} $
given for some $L \in \Bbb Z_+$ by a block map
$$
{\Phi}: X_{[-L, L] } \to\widetilde \Sigma.
$$
Then $X$ has an RI-mapping  $\Psi^{(+)} : X_{[-L,L]} \to X_{[1,L+1]}$.
\end{lemma}
\begin{proof}
Choose a mapping $\widetilde {\Psi}^{(+)}$ that selects for every  
$\widetilde  \sigma \in \widetilde  \Sigma$ an element  of $\widetilde  {X}_{ [1,  L 
+1 ]} \cap  \omega ^{+}(\widetilde \sigma)$, and set
$$
\Psi^{(+)} = \widetilde \Phi \widetilde  \Psi^{(+)}  \Phi.
$$
We show  that $\Psi^{(+)}$ is 
an $RI$-mapping of X.
To see that $\Psi^{(+)}$ satisfies  $(RIa)$, let 
$a\in X_{[-L,L]}$ and let  $x^{-} \in \Gamma^{-}(a)$. Then one has for
$$
\widetilde  \sigma =  \Phi (a), 
\quad\widetilde  x^{ -} = \Phi (x^{-},a), 
$$
that 
$$
\widetilde  x^{-}\in \Gamma^{-}(\widetilde  \sigma),$$
and therefore
$$
(\widetilde  x^{ -}, \widetilde  \sigma,\widetilde \Psi^{(+)}( \widetilde  \sigma )) \in \widetilde X,
$$
and from
\begin{align*}
\widetilde \Phi (\widetilde  {x}^{-},\widetilde  \sigma, \widetilde \Psi^{(+)}( \widetilde  \sigma )) = 
({x}^-, a_{[-L, 0]}, \Psi^{(+)}(a)), 
\end{align*}
one has then, that
$$
\Psi ^{(+)}(a) \in \Gamma^{+}(x^{-}, a_{[-L, 0]}), 
$$
and condition $(RIa)$ is shown.

To show $(RIb)$, let 
$$
a\in  X_{[-L, L]}, \qquad b\in X_{(-3L. L)} \cap\Gamma^-(a).
$$
By $(RIa)$
$$
\Phi (b, a) \in \Gamma ^-(  \Psi^{(+)}( a) ),
$$
and one has that
$$
 \widetilde  \Phi(\Phi(b, a),  \widetilde  \Psi^{(+)}(\Phi( a))) =
  ( b_{(-2L, -L)  }, a_{[-L, 0]},   \Psi^{(+)}( a)     ),
$$
and therefore
$$
\Phi(  b_{(-2L+ \ell,- L) }        , a_{[-L, \ell]}   ) = 
\Phi(  b_{(-2L+ \ell,- L) }        , a_{[-L, 0]}, \Psi^{(+)}( a)_{[1, \ell]}), \quad 1 \leq \ell \leq L,
$$
which implies $(RIb)$. 
\end{proof}

Let $X \subset  \Sigma ^{\Bbb Z}$ be a subshift with an $RI$-mapping $
\Psi^{(+)} : X_{[-L,L]} \to X_{[1,L+1]}.$
Define a block map
$$
\Theta^{(+)}_{\Psi^{(+)}} :  X_{[-L,L]}  \to X_{[-L,0]} \times X_{[1,L+1]} 
$$ 
by 
$$
\Theta^{(+)}_{\Psi^{(+)}}(d) = (d_{[-L,0]}  , \Psi^{(+)} (d)),\quad d \in X_{[-L,L]}.
$$
The block map $\Theta^{(+)}_{\Psi^{(+)}}$ determines an embedding
 $$
  \vartheta^{(+)}_{\Psi^{(+)}}:X\to (X_{[-L,0]} \times X_{[1,L+1]}) ^{\Bbb Z}
  $$
 by
 $$
\vartheta^{(+)}_{\Psi^{(+)}}(x) = (x_{[i-L,i]},  \Psi^{(+)}( x_{[i-L,i+L]}  ))_{i\in \Bbb Z}, \qquad
 x \in X.
 $$
 
 \begin{lemma}
 Let $X \subset  \Sigma ^{\Bbb Z}$ be a subshift, let $L \in \Bbb Z_+$, and let 
  $$
 \Psi^{(+)}:X_{[-L, L]}\to X_{[-L. 0]}\times X_{[1, L+1]},
 $$
 be an  $RI$-mapping of $X$.
 Then $ \vartheta^{(+)}_{\Psi^{(+)}}(X)$ is right instantaneous.
 \end{lemma}
  \begin{proof}
  Let $(b, c) \in \Theta^{(+)}_{\Psi^{(+)}}(X_{[-L, L]} )$. One has by $(RIb)$ that
  \begin{align*}
  \Gamma^-(b,c) &= \{(\Theta^{(+)}_{\Psi^{(+)}}(x^-_{[i-L,i+ L]})_{-\infty < i < 0}: x^- \in X_{(-\infty, L]} , (b, c)= \Theta^{(+)}_{\Psi^{(+)}}( x_{[-L, L]}  )     \}         \\
 & = \{(\Theta^{(+)}_{\Psi^{(+)}}(x^-_{[i-L,i+ L]})_{-\infty < i < 0}: x^- \in X_{(-\infty, L]} , 
 x_{[-L, L]}=(b, c_{[1, L]})  \},
   \end{align*}
  and by $(RIb)$ one has for $x^- \in X_{(-\infty, L]} $, 
  such that $  x_{[-L, L]}=(b, c_{[1, L]})$, that
  $$
(  x^-_{(-\infty, 1]}, \Psi^{(+)}( b, c_{[1, L]}  )) \in X_{(-\infty, L+1]}.
  $$
 It follows for
 $$
 b^\prime= (b_{(-L, 0]}, c_1) , \qquad    c^\prime= \Psi^{(+)}(b^\prime, c_{[1, L+1]}),
 $$ 
 that
 $$
 ( b^\prime, c^\prime) = \Theta^{(+)}_{\Psi^{(+)}}(b^\prime , c  ) \in \omega^+_1(b, c_{[1, L+1]}). \qed
 $$
\renewcommand{\qedsymbol}{}
\end{proof}

We say that a subshift $X \subset  \Sigma^{\Bbb Z}$, for $L\in\Bbb Z_{+} $ has Property $RI$ if it has an $RI$-mapping. By the following theorem Propety $RI$ is an invariant of topological conjugacy.

\begin{theorem}
A subshift has a right instantaneous presentation if and only if it has Property $RI$.
\end{theorem}
\begin{proof}
Apply Lemma (6.2) and Lemma (6.3)
\end{proof}

For a subshift $X 
\subset  \Sigma^{\Bbb Z}$, for $L\in\Bbb Z_{+}$, and for mappings
$$
\Psi^{(+)}: X_{[-L,L] } \to X_{[1,L+1] } 
$$
we have a condition $(LI)$, that is symmetric to the Condition $(RI)$:
\begin{align*}
(LIa): \Psi^{(-)} (a) \in \Gamma ^{-} (a_{[0, L]},x^{+}),  \ \   \quad a &\in X_{[-L,L]},x^{+}
\in \Gamma ^{+}(a).
\\
(LIb):\Psi ^{(-)}(a_{[-\ell, L]},b_{(L, 2L -\ell )} ) = 
\Psi ^{(-)}&(\Psi ^{(-)}(a)_{[- \ell, -1]},a_{[0, L]},b_{(L , 2L - \ell  )} ), 
\\
&b \in X_{(L, 2L)}, a \in X_{[-L,L]}, 1 \leq \ell \leq L.
\end{align*}
We say that a mapping $ \Psi^{(-)} : X_{[-L,L]} \to X_{[1,L+1]}, L\in \Bbb Z_+,$ that satisfies
Condition  $(LI)$  is an $ LI$-mapping of $X$. The symmetric analogues of  Lemma (6.2), Lemma (6.3) and Theorem (6.4) have the symmetric proofs.

We say that a pair $(\Psi_{-},\Psi_{+})$ that consists of an $ LI$-mapping
$\Psi_{-}$ and a $LR$- mapping $\Psi_{+}$  is a pair of $BI$-mappings if also 
the following condition $(BI)$, which comes in two parts, is satisfied:
\begin{align*}
(LBI): \Psi ^{(-)}(b_{(-2L + \ell ,-L )} ,a_{[-L, \ell]})& = \Psi ^{(-)}(b_{(-2L + \ell ,-L )} ,
a_{[-L, 0]},\Psi ^{(+)}(a)_{[1. \ell]}), 
\\
&b \in X_{(-2L, L)}, a \in X_{[-L,L]}, 
\\
(RBI) :\Psi ^{(+)}(a_{[-\ell, L]},b_{(L, 2L -\ell )} ) = 
& \ \Psi ^{(+)}(\Psi ^{(-)}(a)_{[- \ell, -1]},a_{[0, L]},b_{(L , 2L - \ell  )} ), 
\\
&b \in X_{(L, 2L)}, a \in X_{[-L,L]}, 1 \leq \ell \leq L.
\end{align*}

\begin{lemma}
Let 
 $\widetilde{X} \subset \widetilde{\Sigma} ^{\Bbb Z}$ be a bi-instantaneous subshift, and let
 $\varphi $ be a topological conjugacy of $\widetilde{X}$ onto a subshift $X\subset \Sigma^{\Bbb Z}$, 
 that is given by a one-block map 
$\widetilde\Phi : \widetilde{\Sigma}\to \Sigma$
with  $\varphi^{-1} $
given for some $L \in \Bbb Z_+$ by a block map
$$
{\Phi}: X_{[-L, L] } \to\widetilde \Sigma.
$$
Then $X$ has a a pair $(\Psi^{(-)} ,\Psi^{(+)})$ of  BI-mappings,
\begin{align*}
&\Psi^{(-)} : X_{[-L,L]} \to X_{[-L - 1,-1]},
\ 
&\Psi^{(+)} : X_{[-L,L]} \to X_{[1,L+1]}.
\end{align*}
\end{lemma}
\begin{proof}
Choose simultaneously a  mapping $\widetilde{\Psi}^{(-)}$ that selects for every $\widetilde {\sigma} \in
\widetilde {\Sigma}$ an element of $\widetilde{X}_{[-L-1,0)} \cap \omega^{-}(\widetilde{\sigma})$ and a mapping $\widetilde{\Psi}^{(+)}$ that selects 
 for every  
$\widetilde \sigma \in \widetilde \Sigma$ an element of $\widetilde {X}_{ [1,  L +1 ]} \cap  \omega ^{+}(\bar \sigma)$,  and set
$$
\Psi^{(-)} = \widetilde\Phi {\widetilde\Psi^{(-)}} {\Phi},
\qquad
\Psi^{(+)} = \widetilde \Phi \bar {\Psi}^{(+)} \Phi.
$$
One confirms that $(\Psi^{(-)} , \Psi^{(+)})$ is a a pair of   $BI$-mappings for $X$.
\end{proof}

Let $X \subset  \Sigma ^{\Bbb Z}$ be a subshift with a pair $(\Psi^{(-)} , \Psi^{(+)})$,
\begin{align*}
&\Psi^{(-)} : X_{[-L,L]} \to X_{[-L - 1,-1]},
\ 
&\Psi^{(+)} : X_{[-L,L]} \to X_{[1,L+1]},
\end{align*}
 of   $BI$-mappings for $X$.
Define a block map
$$
\Theta_{(\Psi^{(+)},\Psi^{(+)})} :  X_{[-L,L]}  \to X_{[-L-1,-1]} \times  \Sigma \times X_{[1,L+1]} 
$$  
by 
$$
\Theta_{(\Psi^{(+)},\Psi^{(+)})} (a)= (\Psi^{(-)}(a) , a_0 , \Psi^{(+)}( a)),\qquad a \in X_{[-L,L]}.
$$
The block map $\Theta_{(\Psi^{(+)},\Psi^{(+)})} $ determines an embedding
 $$
  \vartheta_{(\Psi^{(+)},\Psi^{(+)})}:X\to (X_{[-L,0]}\times \Sigma \times X_{[1,L+1]}) ^{\Bbb Z}
  $$
 by
 $$
\vartheta_{(\Psi^{(+)},\Psi^{(+)})}(x) =
 (  \Psi^{(-)}( x_{[i-L,i+L]}), x_i,  \Psi^{(+)}( x_{[i-L,i+L]}))_{i\in \Bbb Z}, \qquad
 x \in X.
 $$

The proof of the following lemma  
is a twofold repeat  of
the preceeding one, that uses the components of the pair of $BI$-mappings simultaneously.

\begin{lemma}
Let $X \subset  \Sigma ^{\Bbb Z}$, let $L \in \Bbb Z_+$, and let
$(\Psi^{(-)} , \Psi^{(+)})$,
\begin{align*}
&\Psi^{(-)} : X_{[-L,L]} \to X_{[-L - 1,-1]},
\ 
&\Psi^{(+)} : X_{[-L,L]} \to X_{[1,L+1]},
\end{align*}
be a pair of BI-mappings of X.
 Then $ \vartheta_{(\Psi^{(+)},\Psi^{(+)})}(X)^{\langle [1,L] \rangle}$ is bi-instantaneous.
 \end{lemma}
\begin{proof}
Let $a \in  X_{[-L,L]}$ and set
$$
(b^-(i), d_i, b^+(i)) _{-L \leq i \leq 0}= (\Theta ( a_{[i-L. i+L]}  ) )_{-L \leq i \leq 0}.
$$
Set
$$
\bar a = (a_{[-2L. 0]} , \Psi^{(+)}(  a_{[i-L. i+L]} ).
$$
By $(LRb)$ and $(LBI)$
$$
(b^-(i), d_i, b^+(i)) = (\Theta ( \bar a_{[i-L. i+L]}  ) )_{-L \leq i \leq 0}. 
$$
It follows by $(LIa)$ that
$$
(\Theta ( \bar a_{[i-L. i+L]}  ) )_{-L \leq i \leq 0} \in \omega^+_1( (b^-(i), d_i, b^+(i)) _{-L \leq i \leq 0} ).
$$
The proof of left instantaneity is symmetric.
\end{proof}

We say that a subshift $X \subset \Sigma^{\Bbb Z}$ has property $BI$ if it has a 
a pair of $BI$-mappings. The following theorem implies that Property $BI$ is an invariant of topological conjugacy. 

\begin{theorem}
A subshift has a bi-instantaneous presentation if and only if it has Property $BI$.
\end{theorem}
\begin{proof}
Apply Lemma (6.5) and Lemma (6.6)
\end{proof}

The coded system (see[BH]) with code
$$
{\mathcal C} = \{ 0\alpha^{n}\beta^{n} :   n \in { \Bbb N}\} ,
$$
is an example of a synchronizing system that has neither
a left instantaneous presentation nor a right instantaneous presentation. We 
give an example of a semisynchronizing (see [Kr3]) non-synchronizing subshift 
that has
 neither
a left instantaneous presentation nor a right instantaneous presentation. For this, take as alphabet the set
$$
\Sigma = \{\bold  1, \alpha_{\lambda}, \alpha_{\rho}, \beta_{\lambda}, 
\beta_{\rho}\} ,
$$
and view $\Sigma$ as a generating set of ${\mathcal D}_{2}$ with relations
$$
\alpha_{\lambda}\alpha_{\rho} = \beta_{\lambda} \beta_{\rho}  = 
\bold 1, \ \alpha_{\lambda} \beta_{\rho} = \beta_{\lambda} \alpha_{\rho}  
= 0.
$$
We let $X$ be the subshift in $\Sigma^{\Bbb Z}$ that  contains all $x 
\in  M_{2}$ that are also
label sequences of bi-infinite paths on the directed 
graph that has vertices $v, v(+), v(-)$, four loops at $v$, one with 
labels 
$\alpha_{\lambda},\alpha_{\rho},\beta_{\lambda},\beta_{\rho}$, a 
loop at $v(-)$ with label $\beta_{\lambda}$, a 
loop at $v(+)$ with label $\beta_{\rho}$, an edge from $v$ to $v(-)$ 
with label $\alpha_{\lambda}$, an 
edge from $v(-)$ to $ v(+)$ with label $\bold 1$,  and an edge from $v(+)$ to 
$v$ with 
label $\alpha_{\rho}$. 
Adding a loop at vertex $v(+)$ that carries the label $\bold 1$  one obtains
a semi-synchronizing non-synchronizing right instantaneous subshift that does
not have a left instantaneous presentation.

\section{Presenting a class of subshifts}

We say that a subshift $X \subset \Sigma^{\Bbb Z}$ is strongly bi-instantaneous, 
if there is an
$R\in \Bbb N$  such that
for  any word $a\in \mathcal L(X)$ of length at least 
$R$, there is a word $c\in \Gamma^+(a) \cap  \Gamma^-(a)$ such that 
  $ ca \in \omega^+(a)  $ and $ac \in \omega^-(a)$.

\begin{proposition}
Let $X \subset \Gamma^{\Bbb Z}$ be a strongly bi-instantaneous topologically transitive subshift. 
Let $X$ have Property (A) with associated semigroup the $\mathcal R$-graph semigroup
$ {\mathcal S} _{\mathcal R} ({\frak P},{\mathcal E}^-,{\mathcal E}^+)$, and let
$X$ have Property $(c)$. Then $X$ has an ${\mathcal S}_{\mathcal R} ({\frak P},{\mathcal E}^-, {\mathcal E}^+)$-presentation.
\end{proposition}
\begin{proof}
Let $R\in \Bbb N$ be such that
for  such that for any word $a\in \mathcal L(X)$ of length at least 
$R$, there is a word $c\in \Gamma^+(a) \cap  \Gamma^-(a)$ such that 
  $ ca \in \omega^+(a)  $ and $ac \in \omega^-(a)$.
Assume that  $X$ has Property $(c)$ with respect to the parameter $Q$.
Let 
$
k_\circ\geq  \max \{R,Q\}
$
be such that every element of $\frak P(X)$ has an $\approx$-representative in $A_{k_\circ}(X\negthinspace)$. 
By the strong bi-instantaneity of $X$ there exists for $a \in \mathcal L_k(X), k \geq k_\circ$, words
$c \in \Gamma^+(a) \cap  \Gamma^-(a)  $ such that 
\begin{align*}
ca \in \omega^+(a), \qquad ac \in \omega^-(a). \tag {7.1}
\end{align*}
Then
\begin{align*}
\Gamma^-(a) = \Gamma^-(a(ca)^q),  \ \
\Gamma^+(a) = \Gamma^+((ac)^qa), \qquad q > 1,
\end{align*}
and, in view of the length of $a$, it follows by Property (c) of $X$ that 
\begin{align*}
\Gamma(a) = \Gamma(a(ca)^q), \ \ q > 1, \qquad
a\in {\mathcal L}_{k}(X), k \geq k_\circ. \tag {7.2}
\end{align*}
This allows to assign to a word $a \in {\mathcal L}_{k}(X), k \geq k_\circ$ a $\approx$-class $\frak p(a) \in \frak P(X)$
that contains the points that carry the bi-infinite concatenation of the word $ac$, which by (7.1) belong to $A(X)$, since, as a consequence of (7.2), the $\approx$-class of these points does not depend on the choice of the word $c$.

For a word $a \in {\mathcal L}_{k_\circ+1}(X)\cup{\mathcal L}_{k_\circ+2}(X)   $ we denote by $a^{(-)}  $ the prefix of $a$ that is obtained by removing the last symbol, and by  $a^{(+)}  $ the suffix of $a$ that is obtained by removing the first symbol. For $a \in {\mathcal L}_{k_\circ+1}(X)$, let
$$
c^{(-)}  \in \Gamma^+(a^{(-)}) \cap  \Gamma^-(a^{(-)}),  \quad c^{(+)}  \in \Gamma^+(a^{(+)}) \cap  \Gamma^-(a^{(+)}) ,
$$
such that
$$
c^{(-)}a^{(-)} \in \omega^+(a^{(-)}) ,  \quad  a^{(-)}c^{(-)} \in \omega^-(a^{(-)}), 
$$
$$
c^{(+)}a^{(+)} \in \omega^+(a^{(+)}), \quad a^{(+)}c^{(+)} \in \omega^-(a^{(+)}), 
$$
and denote by
$y[ c^{(-)} , a, c^{(+)}]$
the point $y \in Y(X)$, where
$
y _{[1,  k_\circ+1]} = a,
$
and where $y _{(-\infty,  0]}$ carries the left infinite concatenation of the word  $a^{(-)} c^{(-)} $ and  
 $y _{( k_\circ+1  ,\infty)}$ carries the right infinite concatenation of the word  $c^{(+)} a^{(+)} $.

Let
$ \zeta$ denote an isomorphism of ${\mathcal S} (X)$ onto $ {\mathcal S} _{\mathcal R} ({\frak P}, {\mathcal E}^-, {\mathcal E} ^+)$.  In view of the length of $a$, again appealing to Property (c) of $X$, one has that  the image of the point $y[ c^{(-)} , a, c^{(+)}]$ under $\zeta$ does not depend on the choice of $c^{(-)} $ and $ c^{(+)}$. As a  consequence one obtains well defined mappings
$$
f^{(-)}    :{\mathcal L}_{k_\circ+1} \to {\mathcal S}^- _{\mathcal R} ({\frak P},{\mathcal E}^-, {\mathcal E} ^+), \quad  f^{(+)}  :{\mathcal L}_{k_\circ+1}\to
{\mathcal S}^+_{\mathcal R} ({\frak P}, {\mathcal E}^-,{\mathcal E}^+),
$$
by setting
$$
 \zeta(y[ c^{(-)} , a, c^{(+)}] ) = f^{(+)}  (a)  f^{(-)}  (a)   , \quad a \in{\mathcal L}_{k_\circ+1}(X).
$$
$X^{\langle [1,k_\circ+2] \rangle}$ has the ${\mathcal S}_{\mathcal R} ({\frak P},  {\mathcal E}^-,  {\mathcal E} ^+)$-presentation $( {\mathcal V}, \Sigma, \lambda)$, where
$$
 {\mathcal V} ={\mathcal L}_{k_\circ+1}(X), \quad
\Sigma={\mathcal L}_{k_\circ+2}(X),
$$
and
\begin{align*}
s( a) =  a^{(-)},  t(a) =a^{(+)}, \ \  \lambda (a) =    f^{(-)}  (a^{(-)} )   f^{(+)}  (a^{(+)} ) ,    \quad \  a \in {\mathcal L}_{k_{\circ+2}}(X). \tag {7.3}
\end{align*}
One has 
\begin{align*}
 {\mathcal V} ( \frak p )= \{a \in {\mathcal L}_{k_\circ+1}(X)  : \frak p(a) =  \frak p  \}, \quad \frak p \in {\frak P}. \tag {7.4}
\end{align*}
By (7.3) and (7.4) we have (G1), (G2) and (G3) satisfied, and (G4) holds by construction,
The surjectivity of $\zeta$ implies (G5), and  the irreducibility of $( {\mathcal V} ,  \Sigma, \lambda).$ follows from the topological transitivity of $X$.
\end{proof}

We introduce a Condition $(SBI)$ on a  subshift $ {X}\subset { \Sigma}^{ \Bbb Z}$ that, for topologically transitive subshifts, also describes a strengthening of bi-instantaneity:

\bigskip
\noindent
 $(SBI)$ For $a \in X_{[1,n]}, n \in \Bbb N,$ there exists an $x\in A(X)\cap Y(X)$ such that
 $x_{[1,n]}=a$ and
 $$
 x_{(- \infty, 0)} \in \omega ^-(a_1), \qquad x_{(n, \infty)} \in \omega ^+(a_n). 
 $$

\begin{lemma}
Let  $ {X}\subset { \Sigma}^{ \Bbb Z},\widetilde{X}\subset\widetilde { \Sigma}^{ \Bbb Z},$ be   bi-instantaneous subshifts and let  $\varphi: \widetilde {X}\to X $ be a topological conjugacy, that is given by a 1-block map $\widetilde {\Phi}: \widetilde { \Sigma}  \to \Sigma $, with  $\varphi^{-1}$ given with some  $L \in \Bbb Z_+$ by a block map 
$\Phi: X_{[-L,L]}  \to \widetilde { \Sigma}$. Assume that $X$ satisfies Condition $(SBI)$. Then
$\widetilde{X}$ also satisfies Condition $(SBI)$. 
\end{lemma}
\begin{proof}
Let $\widetilde{a} \in \widetilde{X}_{[1,n]}, n\in \Bbb N.$
By the bi-intantaneity of $\widetilde{X}$ we can choose
\begin{align*}
&\widetilde{b}^- \in \omega^-_L( \widetilde{a}_1), \qquad
\\
& \widetilde{b}^+ \in  \omega ^+_L( \widetilde{a}_n). \tag {7.5}
\end{align*}
Set
$$
{b}^- = \widetilde\Phi(\widetilde{b}^-), \qquad {a} = \widetilde\Phi(  \widetilde{a}  ),\qquad
{b}^+ = \widetilde\Phi(\widetilde{b}^+ ).
$$
By the assumption that $X$ satisfies $(SBI)$, we can choose
\begin{align*}
&x^- \in   \omega^-_\infty (b^-_1), 
\\  
&x^+ \in  \omega^+_\infty (b^+_L),\tag{7.6}
\end{align*}
such that
\begin{align*}
x^-b^-ab^+ x^+ \in A(X)\cap Y(X). \tag {7.7}
\end{align*}
Let 
$$
\widetilde x^- \in   \Gamma^-_\infty (\widetilde a), \qquad 
\widetilde x^+ \in   \Gamma^+_\infty (\widetilde a),
$$
be given by setting
$$
\widetilde x^-\widetilde a\widetilde x^+ = \Phi(x^-b^-ab^+ x^+ ).
$$
By (7.7)
$$
\widetilde x^-\widetilde a\widetilde x^+  \in A(\widetilde X)\cap Y(\widetilde X).
$$
We prove that
\begin{align*}
\widetilde x^+ \in \omega^+_\infty (\widetilde a_n).\tag {7.8}
\end{align*}
For this let
$$
\widetilde d \in \Gamma^-(\widetilde a_n), \qquad
 \widetilde d^\prime \in   \Gamma^-_L(\widetilde d\thinspace \widetilde a_n ).
$$
By (7.5)
$$
 \widetilde d^\prime\widetilde d \widetilde{b}^+ \in \mathcal L(\widetilde X).
$$
Therefore one has for
$$
d^\prime = \widetilde \Phi ( \widetilde d^\prime), \qquad   d =  \widetilde\Phi (\widetilde d  ), 
$$
that
$$
d^\prime d {b}^+ \in \mathcal L( X).
$$
By (7.6)
$$
d^\prime d {b}^+x^+ \in  \mathcal L( X),
$$
which implies that
$$
\Phi( d^\prime db^+x^+  ) =  \widetilde d \widetilde x^+ \in  \mathcal L(\widetilde X),
$$
and (7.8) is proved. The proof of 
$
\widetilde x^- \in \omega^-_\infty (\widetilde a_1)
$
 is symmetric.
\end{proof}

For a finite directed graph $\mathcal G(\frak P, \mathcal E)$ an 
$ {\mathcal S}(\frak P,{\mathcal E}^-,{\mathcal E}^+)$-presentation has Property $(A)$, if and only if  every vertex $\frak p \in \frak P$ has at least two incoming edges (see \cite[Section 2]{HK2}).

\begin{lemma}
Let $\mathcal G(\frak P, \mathcal E)$ be a finite directed graph such that every vertex $\frak p \in \frak P$ has at least two incoming edges, 
and let $X(  \mathcal V, \Sigma, \lambda)$ be an 
$ {\mathcal S}(\frak P,{\mathcal E}^-,{\mathcal E}^+)$-presentation.
Let  $\widetilde{X}\subset\widetilde { \Sigma}^{ \Bbb Z}$ be a  bi-instantaneous subshift
that is topologically conjugate to $X(  \mathcal V, \Sigma, \lambda)$. Then $\widetilde {X}$ is strongly bi-instantaneous.
\end{lemma}
\begin{proof}
To prove the lemma, it is enough to consider the situation that we have a topological conjugacy 
$\varphi: \widetilde {X}\to X(  \mathcal V, \Sigma, \lambda) $, that is given by a one-block map $\widetilde {\Phi}: \widetilde { \Sigma}  \to \Sigma $, with  $\varphi^{-1}$ given with some  $L \in \Bbb Z_+$ by a block map 
$\Phi: X_{[-L,L]}  \to \widetilde { \Sigma}$, and to show that
for any word $\widetilde a\in \mathcal L(\widetilde X)$  there is a word $\widetilde c\in \Gamma^+(\widetilde a) \cap  \Gamma^-(\widetilde a)$ such that $ \widetilde c\widetilde a \in \omega^+(\widetilde a)  $ and $\widetilde a\widetilde c \in \omega^-(\widetilde a)$.

Inspection shows  that an $ {\mathcal S} _{\mathcal R} (\frak P,{\mathcal E}^-,{\mathcal E}^+)$-presentation $X(  \mathcal V, \Sigma, \lambda)$ satisfies $(SBI)$. It follows from  Lemma (7.2) that the subshift $\widetilde{X}$  satisfies $(SBI)$.
Let then $\widetilde a \in \mathcal L( \widetilde{X} )$, and let
\begin{align*}
\widetilde x^- \in \omega^-_\infty (  \widetilde a ),
\\
 \widetilde x^+ \in \omega^+_\infty (  \widetilde a ), \tag {7.9} 
\end{align*}
such that
\begin{align*}
\widetilde x^- \widetilde a \widetilde x^+ \in A(\widetilde{X}) \cap Y(\widetilde{X} ). \tag {7.10}
\end{align*}
Set
$$
x^- = \widetilde \Phi ( \widetilde x^-),\qquad    a = \widetilde \Phi(  \widetilde a ), \qquad
x =  \widetilde \Phi ( \widetilde x^- \widetilde a \widetilde x^+ ).
$$
By (7.10)
\begin{align*}
x^-ax^+ \in A({X}) \cap Y({X} ).
\end{align*}
Denote the periodic point to which $x^-$($x^+ $) is left (right) asymptotic by $p(-)$( $p(+)$). 

The hypothesis on the graph $\mathcal G(\frak P, \mathcal E)$ ensures that an $\mathcal S(\frak P, \mathcal E^-,\mathcal E^+)$-presentation has Property $(A)$. As a consequence of (7.8) there is a $\frak p \in \frak P(X(  \mathcal V, \Sigma, \lambda))$ such that
$$
\{ p(-),p(+)  \} \subset \frak p.
$$
Let $\pi(-)( \pi(+)  )  $ denote a period of $p(-)$($p(+)$). Let $I(-),  I(+) >L,$ be such that
$$
x^-_{(-\infty, I_-]}  = p(-)_{(-\infty, -I_-]} , \qquad x^+_{(n + I_+ ,\infty)}  = p(+)_{(n + I_+ ,\infty)},
$$
and
\begin{align*}
\lambda ( p_{( -I_--\pi(-),  -I_- ]} =  \lambda ( p_{(n + I_+  ,n + I_+ +\pi(+)  ]} )=  \bold 1_\frak p. \tag {7.11}
\end{align*}
Set
$$
b^- =  x_{(-I_-, 0]},  \qquad b^+=  x_{(n , n + I_+]},
$$
By (7.10)
\begin{align*}
\lambda  (b^-ab^+)  =   \bold 1_\frak p. \tag {7.12}
\end{align*}
Both, the source vertex of $b^-$ and the target vertex of $b^+$, are by (7.10) and (7.11) in $\mathcal V(\frak p)$, 
and we can choose a path $b$ in $(\mathcal V, \Sigma, \lambda)$  
such that 
\begin{align*}
s(b)=t(b^+), \quad t(b)=s(b^.), \qquad\lambda (b) = \bold 1_\frak p. \tag {7.13}
\end{align*}

By  (7.10) and (7.13)
$$
x^-ab^+bb^-bx^+ \in X.
$$
   Let a $\widetilde c \in \mathcal L ( \widetilde X  )$ be given by
 $$
 \widetilde x^-\widetilde a \widetilde c\widetilde a \widetilde x^+ = \Phi(x^-ab^+bb^-bx^+).
 $$
 We prove that
\begin{align*}
 \widetilde c\widetilde a \in \omega^+( \widetilde a ). \tag {7.14}
\end{align*}
 For this let
 \begin{align*}
\widetilde d \in \Gamma^-(\widetilde a), \qquad
 \widetilde d^\prime \in   \Gamma^-_L(\widetilde d\thinspace \widetilde a ). \tag {7.15}
\end{align*}
and set
$$
d = \widetilde \Phi (\widetilde d  ), \qquad d^\prime = \widetilde \Phi (\widetilde d ^\prime ),
$$
By (7.9) and (7.15)
$$
\lambda( d ^\prime  d  a b^+) \neq 0,
$$
and then by (7.12 ) and (7.13),
$$
\lambda(d ^\prime  d  a b^+bb^+ab^+ ) \neq 0,
$$
which means that
$$
d ^\prime  d  a b^+bb^-ab^+  \in \mathcal L(X).
$$
It follows that
$$
 \widetilde d \widetilde a \widetilde c\widetilde a= \Phi(d ^\prime  d  a 
 b^+bb^+ab^+ )
\in \mathcal L (\widetilde X ),
$$
and this proves (7.14). The proof of $\widetilde a  \widetilde c\in \omega^-( \widetilde a )$ is symmetric.
\end{proof}

\begin{theorem}
The following are equivalent for a subshift with Property $(A)$
to which there is associated a graph inverse semigroup 
$ {\mathcal S} (\frak P,{\mathcal E}^-,{\mathcal E}^+)$:

 (a) $X$ has an $ {\mathcal S}  (\frak P,{\mathcal E}^-,{\mathcal E}^+)$-presentation,
 
 (b)  $X$ has properties  $(c)$ and $BI$

  and $X$ has a strongly bi-instantaneous presentation,
  
  (c)  $X$ has properties  $(c)$ and $BI$, 
 
 $X$ has a strongly bi-instantaneous presentation, and all bi-instantaneous presentations of $X$ are  strongly bi-instantaneous.
 \end{theorem}
\begin{proof}
One has that $ {\mathcal S}  (\frak P,{\mathcal E}^-,{\mathcal E}^+)$-presentations have Property $(B)$. Apply Proposition 7.1 and Lemma 7.3.
\end{proof}

\begin{corollary}
The following are equivalent for a subshift with Property $(A)$
to which there is associated a graph inverse semigroup 
$ {\mathcal S} (\frak P,{\mathcal E}^-,{\mathcal E}^+)$:

(a) $X$ has an $ {\mathcal S} (\frak P,{\mathcal E}^-,{\mathcal E}^+)$-presentation,

(b) $X$ has property   $(c)$, 
  and $X$ has a pair $(\Psi^{(-)}, \Psi^{(+)})$
  \begin{align*}
&\Psi^{(-)} : X_{[-L,L]} \to X_{[-L - 1,-1]},
\ 
&\Psi^{(+)} : X_{[-L,L]} \to X_{[1,L+1]},
\end{align*}
   of BI-mappings, such that 
  $\vartheta_{(\Psi^{(-)}, \Psi^{(+)})}(X)$
    is
  strongly bi-instantaneous,
  
  (c)  $X$ has  property   $(c)$, \negthinspace
 $X$ has a pair  of BI-mappings, \negthinspace and for all pairs  $(\Psi^{(-)},\negthinspace \Psi^{(+)})$,
\begin{align*}
&\Psi^{(-)} : X_{[-L,L]} \to X_{[-L - 1,-1]},
\ 
&\Psi^{(+)} : X_{[-L,L]} \to X_{[1,L+1]},
\end{align*}
  of BI-mappings of $X$  one has that
  $\vartheta_{(\Psi^{(-)}, \Psi^{(+)})}(X)$
    is
  strongly bi-instantaneous. 
   \end{corollary}
   \begin{proof}
ApplyLemma 6.6 and Theorem 7.4.
\end{proof}

\section{A class of examples}

Let there be given a finite directed graph $\mathcal G(\frak P, \mathcal E)$ with its Markov-Dyck shift
$M\negthinspace D(\mathcal G(\frak P, \mathcal E))$. With  the identity labeling on 
$\mathcal G(\frak P, \mathcal E^-,  \mathcal E^+)$ as $\lambda$ 
we denote for 
$ \frak p \in  \frak P$ by  $\mathcal C_\frak p$ the code that contains the paths 
$c \in \mathcal L(M\negthinspace D(\mathcal G(\frak P, \mathcal E)))$ such that
$
s(c) = \frak p = t(c),  \lambda (c) = \bold 1_\frak p,
$
and such that for every proper prefix $a$ of $c$, $\lambda (a) \neq
 \bold 1_\frak p.$
 
We assume that 
for $\frak q,\frak r \in \frak P  $, one has that $\card (\mathcal E(\frak q , \frak r)) \neq 1 $. 
For $\frak q,\frak r \in \frak P  $ such that $\mathcal E(\frak q , \frak r) \neq \emptyset $
we have the code
$$
\mathcal C_{\frak q , \frak r} = \bigcup _{e \in \mathcal E(\frak q , \frak r)}
e^-\mathcal C_\frak r^\star  e^+,
$$ 
and we denote by $ \mathcal C_{\frak q , \frak r}^\circ$
the code that is obtained by removing from  $ \mathcal C_{\frak q , \frak r}$ all words that have a subword in
 $$
 \bigcup _{\{\frak q^\prime,\frak r^\prime \in \frak P : \mathcal E(\frak q^\prime , \frak r^\prime) \neq \emptyset \}}
 \{e^-(\mathcal C_{ \frak r^\prime}\cup \mathcal C_{ \frak r^\prime}^2) e^+:e \in \mathcal E(\frak q^\prime , \frak r^\prime) \}.
 $$
 We  let $\kappa^-_{\frak q , \frak r, 1}$ denote the identity map on 
 $\mathcal E(\frak q , \frak r)$ and we let $\kappa^-_{\frak q , \frak r, -1}$ denote a permutation of $\mathcal E(\frak q , \frak r)$ without fixed points. Also we denote  by 
 $\mathcal D^\circ_{\frak q , \frak r}$ the code that contains the words 
 $$
 (c^\circ, d) \in \mathcal L(M\negthinspace D(\mathcal G(\frak P, \mathcal E))  \times \{ 1, -1 \}
 $$
 with
 $
 c^\circ \in  \mathcal C_{\frak q , \frak r}^\circ.
 $
Given a word $(c^\circ, d) \in \mathcal D^\circ_{\frak q , \frak r}$ of length $2I, I\in \Bbb N$,  let $K \in \Bbb N$,
 and 
 $
 c^\circ_-,  c^\circ_- \in \mathcal C_{ \frak r},
 $
 and
 $
  c^\circ(k) \in \mathcal C_{ \frak r},  1 \leq k \leq K,
 $
such that  
 $$
 c^\circ =  c^\circ_- (\prod_{1 \leq k \leq K} c^\circ(k))  c^\circ_+,
 $$
and denote by $M_-(c^\circ, d )$($M_+(c^\circ, d )$) the length of $ c^\circ_- $($ c^\circ_+ $), 
and write
 $$
 d = ( d_i)_{1 \leq i \leq 2 I},
 $$
 and set
 $$
 \delta ^-(c^\circ, d ) = d_{M_-(c^\circ, d )}, \qquad  \delta ^+(c^\circ, d) = d_{M_+(c^\circ, d )}.
 $$
 Moreover,  we set
 $$
\mathcal F_{\frak q , \frak r} = \{ (c^\circ, d ) \in \mathcal D^\circ_{\frak q , \frak r}:   
 \delta ^-(c^\circ, d ) \delta ^+(c^\circ, d )                   = -1\}.
 $$
 and we denote by $ \mathcal D_{\frak q , \frak r}^\circ$
the code that is obtained by removing from  $ \mathcal D^\circ_{\frak q , \frak r}$ all words that have a subword in
 $$
 \bigcup _{\{\frak q^\prime,\frak r^\prime \in \frak P : \mathcal E(\frak q^\prime , \frak r^\prime) \neq \emptyset \}}
 \mathcal F_{\frak q^\prime , \frak r^\prime}.
 $$
Finally, we set
$$
 \mathcal D_{\frak p} =
  \bigcup_{\{\frak r \in \frak P \}:
  \mathcal E( \frak p , \frak r ) \neq \emptyset  \}} 
   \mathcal D_{\frak q , \frak r}^\circ, \qquad \frak p \in   \frak P ,
$$
and we denote by $Y$ the coded system of any of the codes
 $ \mathcal D_{\frak p}, \frak p \in  \frak P$, which all yield the same coded system.  
 
\begin{theorem}
The subshift $Y$ is bi-instantaneous and has Property (A) with associated semigroup 
$\mathcal S(\frak P, \mathcal E^-,  \mathcal E^+)$. $Y$  does not have an $\mathcal S(\frak P, \mathcal E^-,  \mathcal E^+)$-presentation.
\end{theorem}
\begin{proof}
The subshift $Y$ inherits bi-instantaneity from $M\negthinspace D(\mathcal G(\frak P, \mathcal E)) \times \{1, -1   \}^\Bbb Z.$

For all $\frak p \in \frak P $ all words in $\mathcal D_\frak p   $ have the same context. Also one observes that $P(A(Y))$ contains precisely the points that carry a bi-infinite concatenation of a code word in  $\bigcup_{\frak p\in \frak P}\mathcal D_\frak p   $. Property $(A)$ of $Y$ follows.

For the proof of
$
\mathcal S(Y) = \mathcal S(\frak P, \mathcal E^-,  \mathcal E^+)
$
one constructs for $\frak q, \frak r \in \frak P, \mathcal E(\frak q, \frak r )\neq \emptyset,$ a representative of 
$e^-, e \in \mathcal E(\frak q, \frak r ),$  from a word
$(b, d) \in \mathcal D_{\frak q, \frak r }$, such that the initial symbol of $b $ is equal to $e^- $
and $\delta(\frak q, \frak r) = 1,$ or from a word
$(b, d) \in \mathcal D_{\frak q, \frak r }$, such that the initial symbol of $b $ is equal to 
$(\kappa^{-1}_{\frak q, \frak r }(e))^{-} $
and $\delta(\frak q, \frak r) = -1,$ by appending the right infinite concatenation of a  word in $\mathcal D_\frak r$ on the right of the prefix of length $M_-(b, d)$ of $(b, d)$, and a left infinite concatenation of a  word in $\mathcal D_\frak q$ on the left of the prefix, and one has the symmetric construction for  a representative of 
$e^+, e \in \mathcal E(\frak q, \frak r ).$

We show that $Y$ does not have Property $(c)$. For this let 
$ \frak q , \frak r \in \frak P, \mathcal E(\frak q , \frak r) \neq \emptyset $, let $M \in \Bbb N$, and let
$(b^\circ, d^\circ) \in \mathcal D_{ \frak q , \frak r}$ be a  word, 
 such that
$M_-(b^\circ, d^\circ ), M_+(b^\circ, d^\circ)  > M.$
 Let $2J$ be the length of $(b^\circ, d^\circ) $.
Let $e^-$ be the initial symbol of $b^\circ$, and let $e^-$ be its final symbol.
Denote by $(b_-, d_-)$
($(b_+, d_+)$) the word that is obtained from the word $(b^\circ, d^\circ)$ by removing the first (last) pair of symbols. Set
$$
(b,d) = (b_-b^\circ b_+, b_-b^\circ b_+),
$$
and write
$$
d = (d_i  )_{1 \leq i \leq 6J - 2}.
$$
Denote for $\delta _-, \delta _+ \in \{1, -1   \} $ by $d(\delta _-, \delta _+)$ the word that is obtained from the word d by simultaneously replacing the symbol $d_{M_-(b^\circ) -1}$, in the case that this symbol is not $\delta _-$, by  $ \delta _-$, and  the symbol $d_{6J - 2 -M_+(b^\circ) +1}$, in the case that this symbol is not $\delta _+$, by  $ \delta _+$. One finds that the words 
$(b, d(\delta _-, \delta _+) ), \delta _-, \delta _+ \in \{1, -1   \} ,$ are admissible for $Y$ and that their left context, as well as their right context is the same. However the words 
$(b, d(\delta _-, \delta _+) ), \delta _-, \delta _+ \in \{1, -1   \} ,$ have distinct contexts. 
Set
$$
e^-(\delta_-) =
 \begin{cases} 
 e^-,  &\text {if $\delta_- = 1,$         }  \\ (\kappa^{-1}_{\frak q, \frak r }(e))^{-}, &\text {if $\delta_- = -1,$      }
  \end{cases}
$$
$$
e^-(\delta_+) =
 \begin{cases} 
 e^-,  &\text {if $\delta_+ = 1,$         }  \\ (\kappa^{-1}_{\frak q, \frak r }(e))^{+}, &\text {if $\delta_+ = -1.$      }
  \end{cases}
$$
One has 
$$
(e^-(\delta^\prime_-), e^+(\delta^\prime_+)) \in \Gamma (b, d(\delta _-, \delta _+) )
$$
if and only if $\delta^\prime_-=\delta _-, \delta^\prime_+= \delta _+$. 
Apply Theorem 4.1 and Theorem 5.1 to conclude the proof.
\end{proof}

One notes that $Y$ actually has Property $(B)$.

\section{Markov codes and zeta functions}

Denoting by
$\Pi_n(X)$ the number of points of period $n$ of a shift-invariant set $X
\subset \Sigma^{\Bbb Z} ,$
the zeta function of $X$ 
is given by
$$
\zeta_X(z) = e^{\sum_{n \in {\Bbb N}} \frac{\Pi_n(X)z^n}{n}}.
$$
We also recall from  \cite{Ke} the notion of a circular Markov code to the extent that is needed here.
We let a Markov code  be given by a  code ${\mathcal C}$ of words
in the symbols of a finite alphabet $\Sigma$
together with a finite set ${ {\mathcal V}}$ and mappings
$t:{\mathcal C} \to  \mathcal V, s:{\mathcal C} \to \mathcal V$.
To a Markov code
$( {\mathcal C} ,t,s)$ there is associated the shift invariant set 
$X_{({\mathcal C},t,s)} \subset \Sigma^{\Bbb Z}$
of points $x \in \Sigma^{\Bbb Z}$ such that 
there are indices 
 $I_k, k \in {\Bbb Z},$  
$$
I_0 \le 0 < I_1,\quad  I_k < I_{k+1}, \quad k \in {\Bbb Z}, 
$$
such that
\begin{align*}
x_{[I_{k},I_{k+1})} \in  {\mathcal C}, \qquad k \in {\Bbb Z},  \tag {9.1}
 \end{align*}
and
\begin{align*}
r(x_{[I_{k-1},I_{k})}) = s(x_{[I_{k},I_{k+1})}), \qquad k \in {\Bbb Z}.  \tag {9.2}
 \end{align*}
$({\mathcal C}, t,s)$ is said to be a circular Markov code
if for every periodic point $x$ in $X_{({\mathcal C},r,s)} $ 
the indices $I_k, k \in {\Bbb Z},$ 
such that  (9.1) 
and  (9.2) hold,
are uniquely determined by $x$.
Given a circular Markov code
$({\mathcal C}, s, r)$ denote by
$
{\mathcal C}(u,w)
$
the set of words $c \in {\mathcal C}$ such that
$s(c) = u$,  $t(c) = w$,$u,w \in {{\mathcal V}}.$
We set
$$
g_{{\mathcal C}(u,v)} =\sum_{0 \leq n <\infty}\card \{ c \in{\mathcal C}: s(c) = u, \ t(c) = v, \ \ell(c) =n\}, 
$$
and we introduce the matrix
$$
H^{({\mathcal C})}(z) = (g_{{\mathcal C}(u,v)}(z))_{u,v \in { {\mathcal V}}}.
$$
For a circular Markov code $({\mathcal C}, t, s),$  one has \cite {Ke} 
\begin{align*}
\zeta_{X_{({\mathcal C}, t, s)}}(z) = \det ( I - H^{({\mathcal C})}(z))^{-1}. 
\tag {9.3}
\end{align*}

Given an ${\mathcal R}$-graph ${\mathcal G}_{\mathcal R}(\frak P,  {\mathcal E}^-,   {\mathcal E}^+ )$. we associate to an
 ${\mathcal S}_{\mathcal R}(\frak P,  {\mathcal E}^-,   {\mathcal E}^+ )$-presentation  $X({\mathcal V} ,\Sigma, \lambda    )$ the state spaces 
 $$
 \Sigma ^- = \{ \sigma \in \Sigma: \lambda (\sigma) \in{\mathcal S}^- \cup\{1_{\frak p}: \frak p \in {\frak P}\}\},
 $$
 and
 $$
  \Sigma ^+ = \{ \sigma \in \Sigma: \lambda (\sigma) \in \{{\bold 1}_{\frak p}: \frak p \in \frak P\} \cup {\mathcal S}^+ \},
 $$
together with  the 0 - 1 transition matrices
 $
( A^-(\rho, \tau))_{\rho, \tau \in \Sigma^-} 
 $
 and
 $
 (  A^+(\rho, \tau)) _{\rho, \tau \in \Sigma^+} ,
 $
 where for $\rho, \tau $ in  $\Sigma ^-$$(\Sigma^+) $  we set
 $
 A^-(\rho, \tau)$$  (A^+(\rho, \tau))
 $
equal to $1$ if and only if $r(\rho) = s(\tau)$.
 We denote the (possibly empty) topological Markov shift with state space $\Sigma^-$($ \Sigma^+$ ) and transition matrix $A^-   $$( A^+   ) $ by $X( \Sigma^-  , A^-)$($X( \Sigma^+ ,  A^+)$).
Also we associate to the finite directed labeled graph $( {\mathcal V}  ,\Sigma, \lambda    )$ the circular Markov code $({\mathcal C}^0(  {\mathcal V}   ,\Sigma, \lambda), r, s)$, where
 ${\mathcal C}^0( {\mathcal V}  ,\Sigma, \lambda)$ ist the set of words
 $$
(\sigma_i)_{1 \leq i \leq I}\in{\mathcal L}(X ( {\mathcal V} , \Sigma, \lambda)  ),\quad I > 1,
 $$
 such that
\begin{align*}
&\lambda( (\sigma_i)_{1 \leq i \leq I}  )
 \in \{{\bold 1}_{\frak p}: 
 \frak p \in \frak P   \},
\\
&\lambda( (\sigma_i)_{1 \leq j\leq J}  )
 \notin \{{\bold 1}_{\frak p}: \frak p \in \frak P   \},\quad 1 < J < I,
 \end{align*}
 we let 
 $$
{\mathcal C}^-_{\circ} ({\mathcal V}, \Sigma, \lambda)\quad ({\mathcal C}^+_{\circ} ({\mathcal V}, \Sigma, \lambda) ) 
 $$
 be the set of words
 $$
(\sigma_i)_{1 \leq i \leq I}\in{\mathcal L}(X ({\mathcal V}, \Sigma, \lambda)  ),\quad I > 1,
 $$
such that
\begin{align*}
& \lambda( (\sigma_i)_{1 \leq i \leq I} ) \in 
{\mathcal S}^-_{\mathcal R}(\frak P,  {\mathcal E}^-,   {\mathcal E}^+ ) \cup \{{ \bold 1}_{\frak p}: \frak p \in \frak P  \},\\
& \lambda( (\sigma_i)_{J \leq i \leq I} ) \in
{\mathcal S}^+_{\mathcal R}(\frak P,  {\mathcal E}^-,   {\mathcal E}^+ ) ,\quad 1 < J \leq I,
\\
& \lambda( (\sigma_i)_{1 \leq i \leq I} ) \in  \{ {\bold 1}_{\frak p}: \frak p \in \frak P  \}\cup{\mathcal S}^+_{\mathcal R}(\frak P,  {\mathcal E}^-,   {\mathcal E}^+ ) ,\\
& \lambda( (\sigma_i)_{1 \leq i \leq J} ) \in{\mathcal S}^+_{\mathcal R}(\frak P,  {\mathcal E}^-,   {\mathcal E}^+ ) ,\quad 1 \leq J<I,
 \end{align*}
and we associate to the finite directed labelled graph $(  {\mathcal V}  ,\Sigma, \lambda    )$ the circular Markov codes $({\mathcal C}^-(  {\mathcal V}  ,\Sigma, \lambda), t, s)$ and $( {\mathcal C}^+({\mathcal V} ,\Sigma, \lambda), t, s)$, where
 ${\mathcal C}^-( {\mathcal V}  ,\Sigma, \lambda)$( ${\mathcal C}^+( {\mathcal V}  ,\Sigma, \lambda)$) ist the set of words
 that contains the words  that are in $  {\mathcal C}^-_{\circ} ( {\mathcal V} , \Sigma, \lambda)$($ {\mathcal C}^+_{\circ} ( {\mathcal V} , \Sigma, \lambda)  $) or that are concatenations of a word in 
$
{\mathcal C}^-_{\circ} ( {\mathcal V} , \Sigma, \lambda)$$ ({\mathcal C}^+_{\circ} ( {\mathcal V} , \Sigma, \lambda) ) 
$
and a word in
$
 {\mathcal L}(\Sigma^-, A^-) 
$
($ 
  {\mathcal L}(\Sigma^+, A^+))
$.

\begin{theorem}
\begin{multline*}
\zeta_{X ({\mathcal V} , \Sigma, \lambda)   }(z) =\\
\frac{\det ({\bold 1} - H^{{\mathcal C}^0({\mathcal V} , \Sigma, \lambda))}   (z))}
{\det ({\bold 1} - A^-z)\det ({\bold 1} - H^{{\mathcal C}^-({\mathcal V}, \Sigma, \lambda))}) (z))\det {(\bold 1} - H^{{\mathcal C}^+({\mathcal V}, \Sigma, \lambda))} (z)))\det ({\bold 1} - A^+z)}.
\end{multline*}
\end{theorem}
\begin{proof}
Apply the formula for the zeta function of a topological Markov shift (\cite {LM}), and note that  $ \det (\bold 1 - A^-z) = 1 $($\det (\bold 1 - A^+z)=1$) if $X(\Sigma^-, A^-)  $( $X(\Sigma^+, A^+) $) is empty. Apply  formula (9.3) and collect the contributions to the zeta function of $X
 ({\mathcal V}, \Sigma, \lambda)   $.
\end{proof}

Following \cite{Ke}, special cases of Theorem 7.1 appeared in \cite {I}, \cite {KM2} and \cite {IK1}.

We denote for an ${\mathcal S}_{\mathcal R}(\frak P, {\mathcal E}^-, {\mathcal E}^+ )$-presentation 
$X({\mathcal V},\Sigma   ,\lambda  )$, and for $\frak p \in \frak P$ by $P_{\frak p}(X({\mathcal V},\Sigma   ,\lambda  ))$ the set of periodic points of $X({\mathcal V},\Sigma   ,\lambda  )$ that carry for some 
$V\in {\mathcal V}$ a bi-infinite concatenation of a path $b$ from $V$ to $V$ such that $\lambda (b) = {\bold 1}_{\frak p}$.

\begin{proposition}
Let
 ${\mathcal G}_{\mathcal R}(\frak P,{\mathcal E}^-,  {\mathcal E}^+ )$, be an ${\mathcal R}$-graph, in which every vertex has at least two incoming edges. Let                  
$X( {\mathcal V} ,\Sigma  ,\lambda )$ and  $X(\widetilde{\mathcal V},\tilde\Sigma ,  \widetilde{\lambda}  )$  be topologically conjugate
${\mathcal S}_{\mathcal R}(\frak P,  {\mathcal E}^-,  {\mathcal E}^+ )$-presentations. Then
$$
\prod_{\frak p \in \frak P} (\xi - \zeta_{P_{\frak p}(X({\mathcal V},\Sigma   ,\lambda  ))} (z)) =
\prod_{\frak p \in \frak P} (\xi - \zeta_{P_{\frak p}(X(\tilde{\mathcal V},\tilde\Sigma ,  \tilde\lambda  ))} (z)).
$$ 
\end{proposition}
\begin{proof}
The hypothesis on the ${\mathcal R}$-graph  $ {\mathcal G}_{\mathcal R}(\frak P, {\mathcal E}^-,  {\mathcal E}^+ )$ implies that
$$
P(A(X( {\mathcal V},\Sigma   ,\lambda  )  )) = \bigcup_{\frak p\in \frak P}P_{\frak p}(X( {\mathcal V},\Sigma   ,\lambda  )),
$$
and, moreover, that this is in fact the partition of $P(A(X( {\mathcal V},\Sigma   ,\lambda  )  )) $ into its $\approx$-equivalence classes, which is invariant under topological conjugacy.
\end{proof}

In the case of graph inverse semigroups of finite directed graphs, in which every vertex has at least two incoming edges, the set $P(A(X( {\mathcal V},\Sigma   ,\lambda  )  )) $ appeared in \cite {HI, HIK} as the set of neutral periodic points. For the case of a Markov-Dyck shift X the coefficient sequence of the Taylor expansion of the coefficient of $\xi^{\card (\frak P) - 1}$ in
 $\prod_{\frak p \in \frak P} (\xi - \zeta_{P_{\frak p}(X)} (z))$ was introduced in \cite {M3} as a generalization of the Catalan numbers. (With the Catalan numbers  $C_k =\frac{1}{n+1} \binom {2n}{n}$ this sequence is in  the case of the Dyck shift $D_N, N > 1,$   equal to
 $
 N^k C_k , k \in \Bbb Z_+.
 $)

\bigskip

\par\noindent Institut f\"ur Angewandte Mathematik 
\par\noindent Universit\"at Heidelberg
\par\noindent Im Neuenheimer Feld 294, 
\par\noindent 69120 Heidelberg, Germany
\par\noindent krieger@math.uni-heidelberg.de


\begin{thebibliography}{99999}

\bibitem [AH]{AH}
{\sc C. J. Ash, T. E.  Hall}
{\it Inverse semigroups on graphs},
Semigroup Forum 
{\bf 11}
(1975)
140--145

\bibitem[BBD1] {BBD1}
{\sc M.-P.~B\'eal, M.~Blockelet, C.~Dima},
{\it Sofic-Dyck shifts},
Mathematical Foundations of Computer Science 2014, Lecture Notes in Computer Science
{\bf 8634},
63--74,
Springer 2014

\bibitem[BBD2] {BBD2}
{\sc M.-P.~B\'eal, M.~Blockelet, C.~Dima},
{\it Sofic-Dyck shifts},
arXiv:1305.7413 [cs.FL]
{\bf  }
(2015)

\bibitem[BBD3] {BBD3}
{\sc M.-P.~B\'eal, M.~Blockelet, C.~Dima},
{\it Finite-type-Dyck shift spaces},
arXiv:1311.4223 [cs.FL]
{\bf  }
(2013)

\bibitem[BH] {BH}
{\sc F.~ Blanchard, G.~ Hansel},
{\it  Syst{\`e}mes cod{\'e}s},
 Theoret. Comput. Sci.
{\bf 44}
(1986)
17--49

\bibitem [CM]{CM}
{\sc T.~ M.~ Carlsen, K. Matsumoto},
{\it Some remarks on the C*-algebras associated to subshifts},
 Math. Scand.
{\bf 95}
(2004)
145 -- 160

\bibitem [CS]{CS}
{\sc A.~ Costa, B.~Steinberg},
{\it A categorical invariant of flow equivalence of shifts},
Ergod. Th. \& Dynam. Sys.,
doi:10.1017/etds. 214.74

\bibitem [CK]{CK}
{\sc J.~Cuntz, W.~Krieger},
{\it A class of $C^*$-algebras and topological Markov chains},
Inventiones Math.
{\bf 56}
(1980)
251 -- 268

\bibitem[HI] {HI}
{\sc T.~Hamachi, K.~Inoue},
{\it Embeddings of shifts of finite type into the Dyck shift},
Monatsh. Math.
{\bf  145}
(2005)
107 -- 129            

\bibitem[HIK] {HIK}
{\sc T.~Hamachi, K.~Inoue, W.~Krieger},
{\it Subsystems of finite type and semigroup invariants of subshifts},
J. reine angew. Math.
{\bf  632}
(2009)
37 -- 61         

\bibitem[HK1] {HK1}
{\sc T.~Hamachi, W.~Krieger},
{\it On certain subshifts and their associated monoids},
Ergod. Th. \& Dynam. Sys.,
doi:10.1017/etds. 214.57

 \bibitem[HK2] {HK2}
{\sc T.~Hamachi, W.~Krieger},
{\it A construction of subshifts and a class of semigroups},
arXiv:1303.4158 [math.DS]
(2013)
                    
\bibitem[I] {I}
{\sc K.~Inoue},
{\it The zeta function, periodic points and entropies of the Motzkin shift},
 arXiv: math.DS/0602100
{\bf  }
(2006)
  
\bibitem[IK1] {IK1}
{\sc K.~Inoue, W.~Krieger},
{\it Subshifts from sofic shifts and Dyck shifts, zeta functions and topological entropy},
 arXiv: 1001.1839 [math.DS]
{\bf  }
(2010)

\bibitem[IK2] {IK2}
{\sc K.~Inoue, W.~Krieger},
{\it Excluding words  from Dyck shifts},
 arXiv:  1305.4720 [math.DS]
{\bf  }
(2013)

          
\bibitem[Ke] {Ke}
{\sc G.~Keller},
{\it Circular codes, loop counting, and zeta-functions},
J. Combinatorial Theory
{\bf  56}
(1991)
75 -- 83         

\bibitem[Ki]{Ki}{\sc B.~P.~Kitchens},
{\it Symbolic dynamics}, Springer, Berlin, Heidelberg, New York
(1998).

\bibitem[Kr1]{Kr1}
{\sc  W.~Krieger},
{\it  On the uniqueness of the equilibrium state},
 Math. Systems Theory
  {\bf  8}
(1974)
97 -- 104

\bibitem[Kr2]{Kr2}
{\sc  W.~Krieger},
{\it  On a syntactically defined invariant of symbolic dynamics},
Ergod. Th. \& Dynam. Sys.
{\bf  20}
(2000)
501 -- 516

\bibitem[Kr3]{Kr3}
{\sc  W.~Krieger},
{\it  On subshifts and topological Markov chains},
  Numbers, information and complexity (Bielefeld 1998),
Kluwer Acad. Publ. Boston MA.
  {}
(2000)
 453 -- 472
 
 \bibitem[Kr4]{Kr4}
{\sc  W.~Krieger},
{\it  On subshifts and semigroups},
  Bull. London Math. Soc.
  {\bf 38}
(2006).
 617 -- 624.

 \bibitem[Kr5]{Kr5}
{\sc  W.~Krieger},
{\it  Presentations of symbolic dynamical systems by labelled directed graphs 
(Notes for a "mini-cours", SDA2, Paris 4-5 October 2007)},
 arXiv:1209.1872 [math.DS]
  {\bf}
 (2012)
 
  \bibitem[Kr6]{Kr6}
{\sc  W.~Krieger},
{\it On flow-equivalence of $\mathcal R$-graph shifts},
  M\"unster J. Math.  {\bf}
 (2015)

 
 \bibitem[KM1]{KM1}
{\sc  W.~Krieger, K.~Matsumoto},
{\it A lambda-graph system for the Dyck shift and its K-groups},
Doc. Math.
  {\bf 8 }
(2003)
 79 -- 96
 
 \bibitem[KM2]{KM2}
{\sc  W.~Krieger, K.~Matsumoto},
{\it  Zeta functions and topological entropy of the Markov-Dyck shifts},
  M\"unster J. Math.
  {\bf  4}
(2011)
 171--184
 
 \bibitem[KM3]{KM3}
{\sc  W.~Krieger, K.~Matsumoto},
{\it  A notion of synchronization of symbolic dynamics and a class of C*-algebras},
  Acta Appl. Math.
  {\bf 126}
(2013)
263--275

 \bibitem[L]{L}
 {\sc M.~V.~Lawson},
{\it Inverse semigroups},
 World Scientific, Singapur, New Jersey, London and Hong Kong
(1998)

\bibitem[LM]{LM}{\sc D.~Lind and B.~Marcus},
{\it An introduction to symbolic dynamics and coding},
 Cambridge University Press, Cambridge
(1995)

\bibitem [M1]{M1}
{\sc K.~Matsumoto},
{\it Stabilized  $C^*$-algebras constructed from symbolic dynamical systems},
Ergod. Th. \& Dynam. Sys.  
{\bf  20}
(2000)
 821 -- 841
 
 \bibitem [M2]{M2}
{\sc K.~Matsumoto},
{\it A simple purely infinite C*-algebra associated with a lambda-graph system of the Motzkin shift},
Math. Z.
{\bf  248}
(2004)
369 -- 394

  \bibitem [M3]{M3}
{\sc K.~Matsumoto},
{\it  C*-algebras arising from Dyck systems of topological Markov chains},
Math. Scand. 
{\bf  109}
(2011)
 31 -- 54
 
 \bibitem [M4]{M4}
{\sc K.~Matsumoto},
{\it Cuntz-Krieger algebras and a generalization of the Catalan numbers},
 Int. J. Math.  
{\bf 24}
(2013)
1350040

 
\bibitem [NP]{NP}
{\sc M.~Nivat, J.-F.~Perrot},
{\it Une g\'en\'eralisation du mono\^{i}d bicyclique},
C.R.Acad.Sc.Paris, Ser. A.
{\bf  271}
(1970)
824 -- 827

\bibitem [P]{P}{\sc J.-E.~Pin},
{\it Syntactic Semigroups. Handbook of Formal
Languages. G.~Rozenberg, A.~Salomaa, Eds., Vol 1, 597 -- 746}, Springer
(1997)

\end{thebibliography}
 \end{document}